\documentclass[10pt]{amsart}

\usepackage[utf8]{inputenc}
\usepackage{amsmath, amsthm, amssymb, amsfonts}

\newcommand\MA{\operatorname{MA}}
\newcommand\Ric{\operatorname{Ric}}
\usepackage[all]{xy}

\newtheorem{theorem}{Theorem}[section]
\newtheorem{lemma}[theorem]{Lemma}
\newtheorem{proposition}[theorem]{Proposition}
\newtheorem{corollary}[theorem]{Corollary}
\newtheorem{definition}[theorem]{Definition}

\newcommand{\C}{{\mathcal C}}
\newcommand{\Q}{{\mathcal Q}}
\newcommand{\LLL}{{\mathcal L}}

\newcommand{\maths}[1]{{\mathbb #1}}  
\newcommand{\RR}{\maths{R}}
\newcommand{\NN}{\maths{N}}
\newcommand{\CC}{\maths{C}}

\newcommand{\ZZ}{\maths{Z}}

\newcommand{\bbb}{{\mathfrak b}}
\newcommand{\aaa}{{\mathfrak a}}
\renewcommand{\ggg}{{\mathfrak g}}

\renewcommand{\lll}{{\mathfrak l}}

\newcommand{\ppp}{{\mathfrak p}}
\newcommand{\hhh}{{\mathfrak h}}

\newcommand{\zzz}{{\mathfrak z}}
\newcommand{\CCC}{{\mathfrak C}}

\newcommand{\ttt}{{\mathfrak t}}

\newcommand{\kkk}{{\mathfrak k}}
\newcommand{\NNN}{{\mathfrak N}}
\newcommand{\XXX}{{\mathfrak X}}
\newcommand{\qqq}{{\mathfrak q}}

\title{Kähler-Ricci flow on horospherical manifold}
\author{\sc{Delgove} François}
\date{\today}

\begin{document}

\maketitle

\begin{abstract}
    In this paper, we prove the existence of a Kähler–Ricci soliton on any smooth Fano horospherical manifold by a study of the Kähler-Ricci flow. Indeed, we prove that the renormalized Kähler-Ricci flow converges in the sense of Cheeger-Gromov and that this limit is a Kähler-Ricci soliton.
\end{abstract}

\section{Introduction}

The founding paper on the Kähler-Ricci solitons is Hamilton's article \cite {Hamilton}. They are natural generalizations of Kähler-Einstein metrics and appear as fixed points of the Kähler-Ricci flow. On a Fano compact Kähler manifold $M$, a Kähler metric $g$ is a \textit{Kähler-Ricci soliton} if its Kähler form $\omega_g$ satisfies :
$$
\operatorname{Ric}(\omega_g) - \omega_g = \LLL_X \omega_g,
$$
where $Ric(\omega_g)$ is the Ricci form of $g$ and $\LLL_X \omega_g$ is the Lie derivative of $\omega_g$ along a holomorphic vector field $X$ on $M$. As usual, we denote the Kähler-Ricci soliton by the pair $(g,X)$ and $X$ is called the \textit{solitonic vector field}. Note that if $X=0$ then $g$ is a Kähler-Einstein metric. When $X \neq 0$, we say that the Kähler-Ricci soliton is \textit{non-trivial}. In order ti recover the original definition, we denote by $(\sigma_t)_{t \in [0,+\infty[}$ the family of diffeomorphisms generated by $\frac{1}{2(1 + t)} X$ and the Kähler form $ \tilde{\omega}= (1+t) \, \sigma_t^* \omega_g$ satisfies the equation of the Ricci flow $\frac{\partial}{\partial t} \, \tilde{\omega_t}= -\Ric(\tilde{\omega}_t)$.

The first study of the solitonic vector field $X$ was done in the paper \cite{TZ1,TZ2}. Thanks to the Futaki function, the authors discovered an obstruction to the existence of Kähler-Ricci soliton and proved that $X$ is in the center of a reductive Lie subalgebra $\eta_r(M)$ of Lie algebra $\eta(M)$ of all holomorphic vector fields. This study also gives us a uniqueness result about Kähler-Ricci soliton (Theorem 0.1 in \cite{TZ1}). Subsequently, the study was developped by Wang and Zhu in \cite{WZ} where they show the existence of Kähler-Ricci solitons on toric manifolds using the continuity method. 
Another wau to prove the existence of a Kähler-Ricci soliton is by a study of the Kähler-Ricci flow. Indeed, in \cite{TZx}, it was proved that if a Fano compact Kähler manifold admits a Kähler-Ricci soliton then the renormalized Kähler-Ricci flow converges in the sense of Cheeger-Gromov to a Kähler-Ricci soliton.  Zhu showed in \cite{ZZZ} that the Kähler-Ricci flow converges to a Kähler-Ricci soliton on a toric manifold without the a priori assumption of existence of Kähler-Ricci soliton. Recently, this technique has been also extended in \cite{2017arXiv170507735H} to toric fibration. In this paper, we extend this result to horospherical manifolds.
\begin{theorem}
Let $M$ be a Fano horosphercal manifold with associated horospherical homegeneous space $G/H$ where $G$ is the complexification of a maximal compact subgroup $K$. The solution $\omega_{t}$ defined for $t \in [0, + \infty[$ of the Kähler-Ricci flow with inital value a $K$-invariant metric $\omega_0$ converges in the Cheeger-Gromov sense to a Kähler form $\omega_{\infty}$ when $t$ tends to $+ \infty$ where $\omega_{\infty}$ is a Kähler-Ricci soliton.
\end{theorem}

Our paper is divided into three sections. The first one makes horospheric geometry reminders, it is inspired \cite{Delcroix2}. The second one gives reminders about the Kähler-Ricci solitons and proves the existence of a solitonic vector field whose Futaki invariant vanishes on every horospherical manifold. Finally, in a last part, the main result of the article is proved by studying the Kähler-Ricci flow.

 The author would like to thank F. Paulin for his help for the redaction of this paper.
 
\section{Horospherical Varieties}\label{blablabla}

In this section, we give some reminders on the theory of algebraic groups and on horospherical varieties. A good reference for the group theory part is \cite{Spr98}. For the notion of horospherical variety, we refer to \cite{Pasquier2009,Tim11}.

\subsection{Reductive group}

Let $G$ be a  reductive connected linear complex algebraic group. We denote by $\ggg$ its Lie algebra. If $K$ is a maximal compact subgroup of $G$ with Lie algebra $\kkk$ then 
$$ \ggg = \kkk \oplus J \kkk, $$
where $ J $ is the complex structure of $ \ggg $. Fix a maximal torus $ T $ of a Borel subgroup $ B $ of $G$. Denote by $ \Phi \subset \XXX (T) $ the root system of $ (G, T) $ where $ \XXX (T) $ is the group of algebraic characters of $ T $. We have the root space decomposition:
$$
\ggg = \ttt \oplus \bigoplus_{\alpha \in \Phi} \ggg_\alpha$$
where for any $\alpha \in \XXX(T)$, $ \ggg_\alpha: = \lbrace x \in \ggg ~:~ \forall h \in \ttt ~ ~ \operatorname{ad}(h) (x) = \alpha (h) x\rbrace$, so that $\ggg_\alpha$ is a complex line if and only if $\alpha \in \Phi$.

Let $ \Phi^+ :=\Phi^+(B)$ be the set of \textit{positive roots} (associated with $ B $) so that the Lie algebra $\bbb$ of $B$ satisfies
$$
\bbb = \ttt \oplus \bigoplus_{\alpha \in \Phi^+} \ggg_\alpha.
$$
We then define the negative roots $\Phi^-:=\Phi^-(B)=-\Phi^+$ of $\Phi$ associated with $B$ so that
$
\Phi = \Phi^+ \sqcup \Phi^-.
$
Let  $B^-$ be the unique Borel subgroup of $G$ called the \textit{opposite Borel subgroup} of $B$ with respect to $ T $ verifying $B \cap B^- =T$. Note that $\Phi^+(B^-)=\Phi^-(B)$.

Denote by $\Sigma$ the set of \textit{simple roots} as the set of roots in $\Phi^+$ that cannot be written as the sum of two elements of $\Phi^+$. For any subset $I$ of $\Sigma$, if we give a subset $I$, let $\Phi_I$ be the subset of $\Phi$ generated by the roots contained in $I$. \textit{The parabolic group containing $B$} with respect to $ I$, denoted by $P:=P_I$, is the connected closed subgroup of $G$ whose Lie algebra is
$$
\ppp = \ttt \oplus \bigoplus_{\alpha \in \Phi^+ \cup \Phi_I} \ggg_\alpha.
$$
Let
$
\Phi_P := \Phi^+ \cup \Phi_I.
$
The parabolic subgroup opposed to $P$, denoted by $Q:=Q_I$, is the parabolic subgroup associated with $I$ for the Borel subgroup $B^-$ i.e. the connected closed Lie subgroup with Lie algebra
$$
\qqq = \ttt \oplus \bigoplus_{\alpha \in \Phi^- \cup \Phi_I} \ggg_\alpha.
$$
Moreover $L = P \cap Q$ is a Levi subgroup of $P$ and we define $\Phi_L$ by
$$
\lll = \ttt \oplus \bigoplus_{\alpha \in \Phi_L} \ggg_\alpha.
$$
. Let $\Phi^+_P$ be the set of roots that are not in $\Phi_Q = \Phi^- \cup \Phi_I$ so that 
$
\Phi^+_P \sqcup \Phi_Q = \Phi.
$
Note that $\Phi^+_P$ is the set of roots of the unipotent radical $U$ of $P$, that
$
\Phi = \Phi^+_P \sqcup \Phi_I \sqcup \Phi_Q^+,
$
and that
$
\Phi^+_Q=-\Phi^+_P.
$
Finally, we define
$$
2\rho_P := \sum_{\alpha \in \Phi^+_P} \alpha.
$$

\subsection{Horospherical subgroups and homogeneous horospherical spaces}

 If $H$ a closed connected algebraic subgroup of $G$, then $H$ is said to be a \textit{horospherical subgroup} of $G$ if $H$ contains the unipotent radical $U$ of a Borel subgroup  $B$. We can build horospherical subgroups using parabolic subgroups. Let $N_G(H):=\lbrace g \in G ~:~ gHg^{-1}=H \rbrace$ be the normalizer of $H$ in $G$.

\begin{proposition}[\cite{pasquierthese}]\label{PHt}
Let $H$ be a horospherical subgroup of $G$. Then $P:=N_G(H)$ is a parabolic subgroup of $G$ containing the Borel subgroup $B$, and the quotient $P/H=T/(T\cap H)$ is a torus.
Conversely, if $H$ is a connected closed algebraic subgroup of $G$ such that $N_G(H)$ is a parabolic subgroup $P$ of $G$ and $P/H$ is a torus, then $H$ is horospherical. The fibration
$$
G/H \longrightarrow G/P.
$$
is a torus fibration over a generalized flag manifold with fiber $P/H$.
\end{proposition}

We obtain the following decomposition of the Lie algebra $\hhh$ of $H$:
$$
\hhh = \ttt_0 \oplus \bigoplus_{\alpha \in \Phi_P} \ggg_\alpha,\text{ where } \ttt_0 = \ttt \cap \hhh.
$$

If $H$ is a horospherical subgroup of $G$, then $G/H$ is called a \textit{homogeneous horospherical space}. On $G/H$, the normalizer $P:=N_G(H)$ acts by right multiplication by the inverse and $H$ acts trivially. We define an action of $G \times P/H$ on $G/H$ by 
$$
(g,pH) \cdot xH = gxp^{-1}H.
$$
Note that the isotropy group of $eH$ is $\lbrace (p,pH) \in G \times P/H ~:~ p \in P \rbrace$. This group is called $diag(P)$ and is isomorphic to $P$ by the first projection.

\subsection{Character groups and one parameter subgroup}
Let 
$ 
\aaa = \ttt \cap J \kkk. 
$
We have an identification between $ \NNN (T) \otimes_\ZZ \RR $ and $\aaa$, where $ \NNN (T) $ is the group of algebraic one-parameter subgroups $\lambda : \CC^\times \rightarrow T$ of $T$, given by the derivative at point $1$ of the restriction of $\lambda$ to $\RR^+_*$. Since, $T \cap H$ is a subtorus of $T$, the image $\NNN(T \cap H)$ defines a sublattice of $\NNN(T)$ corresponding to the one-parameter subgroups having values in $T \cap H$. With $\aaa_0 = \NNN(T \cap H) \otimes_\ZZ \RR$, we have
$
\ttt_0 = \aaa_0 \oplus J \aaa_0.
$

Recall that the Killing form $\kappa$ of $\ggg$ defines a scalar product $ ( \cdot, \cdot) $ on $\aaa \cap[\ggg, \ggg]$. In addition, $\kappa$ is egal to zero on $\zzz(\ggg)$. Thus we can define a global scalar product on $\aaa$ by taking a scalar product invariant by the Weyl group $W$ on $ \aaa \cap \zzz(\ggg) $ and assuming that $ \aaa \cap \zzz(\ggg) $ and $ \aaa \cap [\ggg, \ggg] $ are orthogonal. Let $\aaa_1$ be the orthogonal of $\aaa_0$ for the scalar product $ (\cdot, \cdot) $  so that
$
\ttt = \aaa_0 \oplus \aaa_1 \oplus J \aaa_0 \oplus J \aaa_1
$
and 
$
\ppp/\hhh \simeq \aaa_1 \oplus J \aaa_1
$.

Finally, we recall that there is a natural pairing $ \langle \cdot, \cdot
\rangle $ between $ \NNN (T) $ and $ \XXX (T) $ defined by $ \chi \circ \lambda (z) = z^{\langle \lambda, \chi \rangle} $ for all $\lambda \in \NNN(T), \chi \in \XXX(T)$ and $z \in \CC^\times$. In addition, the natural pairing between $ \Lambda= \XXX (T) \otimes_\ZZ \RR $ and $ \NNN (T) \otimes_\ZZ \RR $ obtained by  $\RR$-linearity can be seen as $ \langle \chi, a \rangle = \ln \chi (\exp a) $ for all $ \chi \in \XXX(T) $ and $ a \in \aaa \simeq \NNN (T) \otimes_\ZZ \RR $. We can also identify $\Lambda$ with $\aaa^*$. Since $(\cdot,\cdot)$ is a scalar product on $\aaa$, for $ \chi \in \XXX (T) $, we denote by $ t_\chi $ the unique element of $ \aaa $ such that 
\begin{equation}{\label{Dada2}}
\forall a \in \aaa, ~~ (t_\chi, a) = \langle \chi, a \rangle.
\end{equation}
For every $\alpha \in \Phi$, let $e_\alpha$ be a generator of the complex line $\ggg_\alpha$ such that $[e_\alpha, e_{-\alpha}] = t_\alpha$. Let us end this section by recalling the polar decomposition.
\begin{proposition}[\cite{Delcroix2}]\label{polar}
The image of $\aaa_1$ in $G$ under the exponential is a fundemental domain for the action of $K \times H$ on $G$, where $K$ acts by multiplication on the left and $H$ by multiplication on the right by the inverse. As a consequence, the set $\lbrace \exp(a)H ~:~ a \in \aaa_1 \rbrace$ is a fundamental domain for the action of $K$ on $G/H$.
\end{proposition}

\subsection{Horospherical variety}

Recall that a complex algebraic  $G$-variety is a reduced finite type scheme over $\CC$ with an algebraic action of $G$. A normal complex algebraic $G$-variety $X$ will be said to be \textit{$G$-spherical} if it admits an open and dense $B$-orbit. Note that $X$ is then connected. A $G$-spherical variety $X$ will be called \textit{horospherical} if the stabilizer $H$ in $G$ of a point $x$ in the open and dense $B$-orbit is horospherical. We then will say that $(X,x)$ is \textit{a horospherical embedding}. In particular, $X$ has $G/H$ as a dense open subset. Two horospherical embeddings $(X,x)$ and $(X',x')$ are \textit{isomorphic } if there is an $G$-equivariant isomorphism from $X$ to $X'$ sending $x$ to $x'$. In this paper, we will assume that $X$ is smooth. Thanks to the GAGA theorems (see \cite{AIF_1956__6__1_0}), the variety $X$ is therefore a Fano projective complex manifold and in particular $X$ is a connected compact Fano Kähler manifold for the metric induced by the Fubini-Study metric of the projective space.

We fix a homogeneous horospherical space $G/H$. Recall that there is a unique element $\alpha^\vee \in \XXX(T)^* \simeq \NNN(T)$ such that $\langle \alpha, \alpha^\vee \rangle=2$. With our previous notations, we have 
$
\alpha^{\vee}= \frac{2}{\Vert t_\alpha \Vert^2} \, t_\alpha.
$
Define $a_\alpha \in \ZZ$ by $ a_\alpha=\langle 2\rho_P, \alpha^\vee \rangle$ and \textit{ Weyl's dominant closed chamber} $\CCC$ by
$$
\CCC = \lbrace \chi \in \Lambda ~:~ \forall \alpha \in \Phi^+ ,~~ \langle \chi, \alpha^\vee \rangle \geq \rbrace.
$$
and \textit{the open dominant Weyl chamber} as the interior of the closed dominant Weyl chamber. The semigroup of  \textit{dominant weights} is defined by $\Lambda^+:= \XXX(T) \cap \CCC$.

We have the following definition introduced in \cite{pasquierthese}.
\begin{definition}\label{GHref}
Let $G/H$ be a homogeneous horospherical space. A convex polytope $\Q$ of $\aaa \simeq \NNN(T) \otimes_\ZZ \RR$ is said to be \rm{$G/H$-reflective} if we have the following three conditions:
\begin{enumerate}
\item[(1)] $\Q$ has its vertices in $\NNN(T) \cup \left\lbrace \frac{\alpha^\vee}{a_\alpha} ~:~ \alpha \in \Phi^+_P \right\rbrace$ and contains $0$ in its interior,
 
\item[(2)] the dual polytope $\Q^*$ has its vertices in $\XXX(T)$,
\item[(3)] for all $ \alpha \in \Phi^+_P$, we have $\cfrac{\alpha^\vee}{a_\alpha} \in \Q$.
\end{enumerate}
\end{definition}

The \textit{moment polytope} $\Delta^+ \subset \aaa^*$ with respect to the Borel subgroup $B$ of the horospherical manifold $X$ is the \textit{Kirwan's moment polytope} of the Kähler manifold $(X,\omega)$ for the action of a maximum compact subgroup $K$ of $G$, where $\omega$ is a $K$-invariant Kähler form in $ 2 \pi \, c_1(X)$ (see \cite{Bri87,K1} for more details). We then have the following result.

\begin{proposition}[\cite{pasquierthese}]\label{roro1}
Let $G/H$ be a homogeneous horospherical space. There exists a bijection between theset of Fano horospherical  embedding of $G/H$ and the set of $G/H$-reflective polytopes $\Q$ in $\aaa$. 
In addition, the polytope $ 2 \rho_P + \Q^*$ is the moment polytope with respect the Borel subgroup $B$ of the horospherical embedding. The assumption $(3)$ of Definition \ref{GHref} is then equivalent to the fact that $ \Delta^+=2 \rho_P + \Q^*$ is included in $\CCC$. In particular, we see that $0 \in \operatorname{Int}(\Q^*)$ and that therefore $2 \rho_P \in \operatorname{Int}(\Delta^+)$.
\end{proposition}

We fix a horospherical embedding $(X,x)$ of $G/H$ and we denote $P:=N_G(H)$. By taking the restriction to the open $B$-orbit, we have an isomorphism between the $G$-equivariant automorphisms of $X$ and those of $G/H$ :
\begin{equation}\label{Aut_G(X)}
\operatorname{Aut}_G(X) \simeq \operatorname{Aut}_G(G/H) \simeq P/H.
\end{equation}
One can consult \cite{Knop91} which deals with the problem in the more general context of spherical varieties.

\subsection{Associated linearized line bundle}
In this section, we introduce associated linearize line bundle over homogeneous horospherical space. They were first introduced by Delcroix in \cite{Delcroix2}. Let $H$ a  be horospherical subgroup of $G$ and $P=N_G(H)$.

\begin{definition}
A line bundle $\Pi: L\rightarrow G/H$ is $(G \times P/H)$-linearized if there exists an action $\Theta : ( G \times P/H) \times L \rightarrow L$ denoted by $(x,y) \mapsto \Theta_x(y)$ such that
\begin{itemize}
\item[$\bullet$] $\Pi : L \rightarrow G/H $ is a $(G \times P/H)$-equivariant morphism,
\item[$\bullet$] the application induced by the action of $\Theta_{{(e,pH)}} $ between the fibers is linear.
\end{itemize}
\end{definition}
Note that to any $(G \times P/H)$-linearized line bundle, we can associate a character of $P$. Indeed, for any $p \in P$, the action of $(p,pH) \in diag(P)$ is trivial on the trivial class in $G/H$. Thus, $\Theta_{(p,pH)}$ induces a linear isomorphism  between $L_{eH}$ and $L_{eH}$ and therefore a linear representation of dimension 1 of $diag(P) \simeq P$ i.e. there is a character $\chi \in \XXX(P)$ such that
\begin{equation}\label{chixi}
(p,ph) \cdot \xi = \chi(p) \xi,~~ \forall \xi \in L_{eH}.
\end{equation}
Let us consider the projection $\pi :G \longrightarrow G/H$. We can define the pulled-back line bundle $\pi^*L$ over $G$. Since $\pi$ is $G$-equivariant, the line bundle $\pi^*L$ admits a $G$-linearization. In particular, we define a global section $s$ on $\pi^*L$ by chosing an element $s(e) \in (\pi^*L)_e$ and setting
\begin{equation}\label{sdef}
s: g \in G \mapsto g \cdot s(e) \in (\pi^*L)_g.
\end{equation}
Now let us consider the inclusion $\iota : P/H \longrightarrow G/H$. We can define the restriction $\iota^*L$ of the line bundle $L$. Since $\iota$ is equivariant for the action of $P \times P/H$, we obtain that $\iota^*L$ is a $(P \times P/H)$-linearized bundle. Two global sections of $\iota^*L$ can also be defined:
\begin{align}\label{sr}
& s_r : pH \in P/H \mapsto (p,eH) \cdot s'(e) ,\\
& s_l : pH \in P/H \mapsto (e,p^{-1}H) \cdot s'(e),
\end{align}
where the element $s'(e) \in (\iota^*L)_e$ is such that $s'(e)$ and $s(e)$ are mapped to the same  element of $L_{eH}$ by the canonical applications $\pi^*L \rightarrow L$ and $\iota^*L \rightarrow L$. Note that these sections are linked by the formula:
$$
\forall p\in P,~~s_r(pH)= \chi(p) s_l(pH).
$$
We have the following commutative diagram
$$
 \xymatrix{
    \pi^*L \ar[r] \ar[d] & L \ar[d] & \iota^*L \ar[l] \ar[d] \\
    G \ar@/^1pc/[u]^s \ar[r]_\pi & G/H & P/H \ar@/^{-1pc}/[u]_{s_r} \ar@/^1pc/[u]^{s_l} \ar[l]^{\iota} 
  }
$$
Given a Hermitian metric $q$ on a complex line bundle $L$ over a complex manifold $M$ and a local trivialisation $s$ of $L$ over an open subset $U$ of $M$, we define the local potential of $q$ with respect to $s$ by $\varphi : x \in U \mapsto - \ln a_x$ where $a_x= \vert s(x) \vert_q^2$. We can associate, to any Hermitian metric $q$, a $(1,1)$-form $ \omega_q $ called the \textit{curvature of $ q $} by $ \omega_q\vert_U = \sqrt{-1} \partial \overline {\partial} \varphi $ where $ \varphi $ is the local potential. One checks that $ \omega_q \vert_U $ does not depend on the local trivialisation and therefore defines a $ (1,1) $-global form. In addition, one can prove that $ \omega_q \in 2  \pi \, c_1 (L) $. We will also say that \textit{$L $ has positive curvature} if there is a metric $ q $ such that $ \omega_q $ is a Kähler form. Fix a reference Hermitian metric $ q ^ 0 $ on $L$ and for any Hermitian metric $ q $, we define a smooth function $ \psi $ on $ X $, called the \textit{global potential of $ q$ with respect to $ q^0 $} by 
$$
\forall x \in X,\forall \xi \in L_x,~~\vert \xi \vert_q^2 = e^{- \psi(x)} \vert \xi \vert^2_{q^0}.
$$
By definition of the $(1,1)$-forms $\omega_q$ and $\omega_{q^0}$ and by computing in local charts, we see that the function $\psi $ satisfies the following relation:
$$
\omega_{q^0} = \omega_q + \sqrt{-1} \, \partial \overline {\partial} \, \psi.
$$
We refer to \cite{Demailly1} for more details.

Let $ G / H $ be a homogeneous horospherical space, $ L $ a $ (G \times P / H )$-linearized line bundle over $ G / H $ and $ q $ a Hermitian $ K $- invariant metric  on $ L $. We can then consider the Hermitian metric $\pi^*q$ on $\pi^*L$ and define the local potential $\phi$ (which is actually defined on the whole $G$) with respect to the section $s$ of Equation \eqref{sdef} :
$$
\phi : g \in G \mapsto -2 \ln \vert s(g) \vert_{\pi^*q} \in \RR.
$$
We also define the potential $ u: \aaa_1 \rightarrow \RR $ with respect to tje section $s$ of $ \iota^*L $ defined in Equation \eqref{sdef}:
\begin{equation}\label{potudef}
u : x \in \aaa_1 \mapsto -2 \ln \vert s_r(\exp(x)H) \vert_{\iota^*q} \in \RR.
\end{equation}
We have the following relation between these two potentials, using the character $\chi$ defined in equation \eqref{chixi}.
\begin{proposition}[Propostion 2.7 in \cite{Delcroix2}]
For all $k \in K$, $x \in \aaa_1$ and $h \in H$, we have
$$
\phi(k \, \exp(x) \, h ) = u(x) - 2 \ln \vert \chi( \exp(x)h ) \vert.
$$
\end{proposition}   

\subsection{Curvature in horospherical case}

Let us now, recall Delcroix's computation of the curvature of a Hermitian metric on a line bundle over $ G / H $ in an adapted basis. The first step is to define this basis. For this, we identify the tangent space of $G/H $ at
$eH $ with $\ggg/\hhh \simeq \oplus_{\alpha \in \Phi^+_P} \aaa_1 \oplus J \aaa_1 \oplus \CC e_{-\alpha}$ . We get a complex basis of the tangent space $ T_{eH} (G / H )$ as the concatenation of a real basis $ (l_j)_{1 \leq j \leq r} $ of $ \aaa_1 $ with $ (e_ {-\alpha})_{\alpha \in \Phi^+_P} $. On $ P / H $, we can define for $ \xi \in T_{eH} (G / H )$ the holomorphic vector field:
\begin{equation}\label{Rxi}
R \xi : pH \mapsto (H,p^{-1}H) \cdot \xi.
\end{equation}
We then have a complex basis of $ T^{1,0} (P / H) $ given by $ (Rl_j - \sqrt {-1} J Rl_j) / $2 for $ 1 \leq j \leq r$ and $ (R e_ {-\alpha} - \sqrt{ -1} JR e _{-\alpha}) / $2 for all $\alpha \in \Phi^+_P$, and we denote by $ (\gamma_j)_{ 1 \leq j \leq r} \cdot (\gamma_\alpha)_{\alpha \in \Phi^+_P} $ the dual basis. 
\begin{theorem}[Theorem 2.8 in \cite{Delcroix2}]\label{delomega}
Let $\omega$ be the curvature $(1,1)$-form of a $K$-invariant Hermitian metric $q$ on a $(G \times P/H)$-linearized line bundle $L$ over $G/H$, whose associated character is denoted by $\chi$. The form $\omega$ is determined by its restriction to $P/H$, given for any $x \in \aaa_1$ by
$$
\omega_{\exp(x)H}= \sum_{1 \leq j_1,j_2 \leq r} \cfrac{1}{4} \cfrac{\partial^2 u}{\partial l_{j_1} \partial l_{j_2}}(x) \sqrt{-1} \, \gamma_{j_1} \wedge \overline{\gamma}_{j_2} + \sum_{\alpha \in \Phi^+_P} \langle \alpha, \frac{1}{2}\nabla u(x) - t_\chi \rangle \sqrt{-1} \, \gamma_\alpha \wedge \overline{\gamma}_\alpha
$$
where $\nabla u(x) \in \aaa_1$ is the gradient of the function $u$ defined by Equation \eqref{potudef} for the scalar product $( \cdot, \cdot)$.
In addition, with $\operatorname{MA}_\RR(u)=\det(\operatorname{Hess}(u))$,
$$
\omega^n_{\exp(x)H} = \dfrac{\MA_\RR(u)(x)}{4^r 2^{\operatorname{Card}(\Phi^+_P})} \prod_{\alpha \in \Phi^+_P} \langle  \alpha , \nabla u(x) -2 t_{\chi}  \rangle \, \Omega,
$$
where $n$ is the dimension of $G/H$ and
$$
\Omega:= \Big( \bigwedge_{1 \leq j \leq r} \gamma_j \wedge\overline{\gamma}_j \Big)  \wedge   \Big( \bigwedge_{\alpha \in \Phi^+_P} \gamma_\alpha \wedge \overline{\gamma}_\alpha \Big).
$$
\end{theorem}
Moreover, by choosing $L=K_{G/H}^{-1}$ $s(e)$ appropriately, with $s_r(e)$ defined in equation \eqref{sr}, we have 
\begin{equation}\label{omegansr}
\omega^n_{\exp(x)H} = \dfrac{\MA_\RR(u)(x)}{4^r 2^{\operatorname{Card}(\Phi^+_P})} \prod_{\alpha \in \Phi^+_P} \langle  \alpha , \nabla u(x) -2 t_{\chi} \rangle s_r^{-1}  \wedge \overline{s_r}^{-1}.
\end{equation}
In order to explain the choice of $s(e)$, let
$$
S:= \Big( \bigwedge_{1 \leq j \leq r} \gamma_i \Big) \wedge \Big( \bigwedge_{\alpha \in \Phi^+_P} \gamma_\alpha \Big), 
$$
which therefore defines a section of the  line bundle $ \iota^*K_{G/H}$. In particular, we have 
$
S \vert_{eH} \in (\iota^*K_{G/H})_{eH},
$
and using the following isomorphisms
$$
(K_{P/H})_{eH} \simeq (\iota^*K_{G/H})_{eH} \simeq (K_{G/H})_{eH} \simeq (\pi^*K_{G/H})_{eH},
$$
we denote via these isomorphisms $s(e):=S \vert_{eH}$. We then obtain, by the definition of $\Omega$ and $s_r$ (Equations \eqref{Rxi} and \eqref{sr}), that
$$
 \forall x \in \aaa_1,~~s_r(\exp(x)H)=S \vert_{\exp(x)H},
$$
and we can conclude. We will constantly use this choice afterwards.

In this section, we consider a horospherical manifold $X$ with reductive group $G$ and horospherical subgroup $H$. Note $P=N_G(H)$ the normalizer of $H$ in $G$. Recall that $P$ is a parabolic subgroup and that there is another parabolic subgroup $Q$ such that there is a Levi subgroup $L$ with $P \cap Q= L$. From now, we consider that the parabolic subgroup $P$ contains the opposite Borel subgroup $B^-$ i.e. $\Phi^+_P= \Phi^- \backslash \Phi_L = - ( \Phi^+ \backslash \Phi_L) $. Recall that we introduced the moment polytope $ \Delta^+ $ of $X$ with respect to the Borel subgroup $B$. Now, we define 
\begin{equation}\label{polytopeDelta} 
\Delta: = -2 \rho_P - \Delta^+
\end{equation}
and \textit{the support function} $v_{2 \Delta}$ by
\begin{equation}\label{fctsupport}
v_{2 \Delta}: x \in \aaa^*  \sup_{p \in    2 \Delta} (x,p)_{e} \in \RR,
\end{equation}
where $(\cdot,\cdot)_e$ is the usual euclidian scalar product on $\aaa^*$ (for any basis of $\aaa_1$). This function satisfies the following properties:
\begin{itemize}
\item[$\bullet$] $\forall x \in \aaa_1^*,~~\forall \alpha \in \RR^+_*,~~ v_{2 \Delta}(\alpha \, x) = \alpha v(x)$
\item[$\bullet$] $\forall (x,y) \in (\aaa_1^*)^2,~~ v_{2 \Delta}(x+y) \leq v_{2 \Delta}(x) + v_{2 \Delta}(y)$
\item[$\bullet$] $\forall x \in \aaa_1^*, ~~ v(x) \leq d \, \Vert x \Vert $ where $d= \sup_{p \in 2 \Delta} \Vert p \Vert$ and $\Vert \cdot \Vert$ the  norm on $\aaa$ associated with the scalar product $(\cdot,\cdot)_{e}$
\end{itemize}
In addition, if $x \in \aaa_1^*$ is such that 
$\Vert x  \Vert=1$, then $2 \Delta$ is contained in the half-space $\lbrace y \in \aaa_1^*~:~ (x,y) \leq v_{2 \Delta}(x) \rbrace$ and at least one point of $2 \Delta$ is in the border of this half-space i.e. in $\lbrace y \in \aaa_1^*~~:~ (x,y) = v_{2 \Delta}(x) \rbrace $.

\begin{proposition}[Proposition $4.2$ in \cite{Delcroix2}]\label{potconv}
Let $ q $ be $K$-invariant Hermitian metric with positive curvature on the line bundle $ K_X^{-1} $ and we denote $\chi$ the character of the restriction of $K_X^{-1}$ in $G/H$ and let $ u: \aaa_1 \rightarrow \RR $ be the convex potential defined by Equation \eqref{potudef}. Then $u $ is a smooth and strictly convex function such that the application $ d u: x \mapsto d_x u \in \aaa_1^*$ verifies $\text{im}(du)=2 \Delta $ and the function $ u-v_{2 \Delta} $ is bounded on $ \aaa_1 $. In particular, the polytope $ 2 \, \Delta$ is independent of the chosen metric $q$.
\end{proposition}
Since $2 \rho_{Q} = -2 \rho_{P}$ belongs to $\operatorname{Int}(\Delta^+)$ by Proposition \ref{roro1}, we have
\begin{equation}\label{0inD}
0 \in \operatorname{Int}(\Delta).
\end{equation}
Recall that
$
\Delta^+ \subset \CCC,
$
where $\CCC$ is the Weyl chamber for the Borel subgroup $B$ (see Proposition \ref{roro1}). Thus 
$$
\forall \alpha \in \Phi^+(B), \forall p \in \Delta^+, ~~ \langle p , \alpha^\vee \rangle \geq 0.
$$
The latter can still be written
$$
\forall \alpha \in \Phi^+(B), \forall p\in  \Delta^+, ~~ ( p, \alpha) \geq 0.
$$
Since $\Phi^+_P \subset \Phi^{+}(B^{-})=-\Phi^{+}(B)$, we then have
$$
\forall \alpha \in \Phi^+_P, \forall p \in -\Delta^+, ~~ ( p, \alpha) \geq 0.
$$
Hence by compactness of $\Delta^+$, there exists $f'>0$ such that
\begin{equation}\label{f>0}
\forall \alpha \in \Phi^+_P, ~~\forall p \in -\Delta^+, ~~ 0 \leq ( \alpha, p ) \leq f'.  
\end{equation}

\section{Kähler-Ricci solitons in the horospherical case}\label{setuphoro}

\subsection{Definition} Let $(M,\omega)$ be a compact Kähler manifold. Recall that the Kähler metric $g$ and the Kähler form $\omega$ can be written locally 
\begin{equation}\label{metkal}
g= g_{i \overline{j}} ~ dz^{i} \otimes d \overline{z}^j
\end{equation}
\begin{equation}\label{forkala}
\omega= \sqrt{-1} \, g_{i \overline{j}} ~ dz^{i} \wedge d \overline{z}^j.
\end{equation}
Now, the \textit{Ricci form} is the real $(1,1)$-form defined locally by
\begin{equation}\label{forTR}
\Ric(\omega)= \sqrt{-1} R_{i \overline{j}} ~ dz^{i} \wedge d \overline{z}^j,~~ R_{i \overline{j}}= - \partial_i \partial_{\overline{j}} \log \det( g_{k \overline{l}}).
\end{equation}
Recall that it can be written globally as
\begin{equation}\label{Ric=omegan}
\Ric(\omega)= \sqrt{-1}\, \partial \overline{\partial} \, \log \omega^n.
\end{equation}
and it satisfies 
\begin{equation}\label{Ricc_1}
\Ric(\omega) \in 2 \pi \, c_1(M),
\end{equation}

Recall that a compact Kähler manifold is called \textit{Fano} if its first Chern Class is positive i.e. $c_1(M)>0$. On a Fano compact Kähler manifold $M$, a Kähler metric $g'$ is a \textit{Kähler-Ricci soliton} if its Kähler form $\omega_{g'}$ satisfies :
$$
\operatorname{Ric}(\omega_{g'}) - \omega_{g'} = \LLL_X \omega_{g'},
$$
where $\LLL_X \omega_{g'}$ is the Lie derivative of $\omega_{g'}$ along a holomorphic vector field $X$ on $M$. Usually, we denote the Kähler-Ricci soliton by the pair $(g',X)$ and $X$ is called the \textit{solitonic vector field}. We immediately note that if $X=0$ then $g'$ is a Kähler-Einstein metric. When $X \neq 0$, we say that the Kähler-Ricci soliton is \textit{non-trivial}.  By abuse, we will say that $g'$ or $\omega_{g'}$ is a Kähler-Ricci soliton if there exists a holomorphic vectors field on $M$ such that $(g,X)$ is Kähler-Ricci soliton.

\subsection{Horospherical case}
Let us fix a connected Fano compact Kähler manifold $M$. Let us recall (see for instance \cite{dercalabi} for details on real and complex vector fields) that the group of complex automorphisms 
$\operatorname{Aut}(M)$ of $M$ is a finite dimensional Lie group whose Lie algebra is the set, denoted $ \eta^\RR(M)$, of holomorphic real vector fields (see Theorem 1.1 in Chapter $3$ of \cite{Kobayashi2}). If $K$ is a maximal compact subgroup of the connected identity component $\operatorname{Aut}^{\circ}(M)$ of $\operatorname{Aut}(M)$ then we have, see \cite{Fujiki1978}, that
$$
\operatorname{Aut}^{\circ}(M) = \operatorname{Aut}_r(M) \ltimes R_u,
$$
where $\operatorname{Aut}_r(M)$ is a reductive subgroup of $\operatorname{Aut}^{\circ}(M)$ and the complexification of $K$ and $R_u$ the unipotent radical of $\operatorname{Aut}^{\circ}(M)$. In addition, if we denote by $\eta^\RR(M)$, $\eta_r^\RR(M)$, $\eta_u^\RR(M)$ and $\kkk$ the Lie algebras of $\operatorname{Aut}(M),\operatorname{Aut}_r(M),R_u$ and $K$ respectively, then we have
$$
\eta^\RR(M)= \eta_r^\RR(M) \oplus \eta_u^\RR(M).
$$
Recall that  $X \in \eta^\RR(M)$ if and only if $\LLL_X \omega$=0. Moreover, if we denote by $\eta(M)$ the Lie algebra of the complex holomorphic vector fields i.e.~the holomorphic sections of the complex vector bundle $T^{1,0} M$, then there is an isomorphism between $\eta^\RR(M)$ and $\eta(M)$ given by $X \mapsto X^{1,0}$. If we denote by $\eta_r(M)$ and $\eta_u(M)$ the image of $\eta_r^\RR(M)$ and $\eta_u^\RR(M)$ respectively by the previous isomorphism, we then obtain a decomposition 
$$
\eta(M)= \eta_r(M) \oplus \eta_u(M).
$$
Assume that $(M,x)$ is a horospherical embedding under the action of the reductive group $\operatorname{Aut}_r(M)=:G$ and that $x \in M$ is such that its isotropy group $H$ in $G$ is a horospherical subgroup in $G$ containing the unipotent radical of the opposite Borel subgroup $B^-$. Let $P:=N_G(H)$. Using the decompositions in section \ref{blablabla}, with $\ggg= \eta^{\RR}_r(M)$ the Lie algebra of $G$, we have
$$
\ggg = \aaa_1 \oplus \aaa_0 \oplus J \aaa_1 \oplus J \aaa_0 \oplus \bigoplus_{\alpha \in \Phi^+_P} \ggg_\alpha \oplus \bigoplus_{\alpha \in \Phi_I} \ggg_\alpha \oplus \bigoplus_{\alpha \in -\Phi^+_P} \ggg_\alpha.
$$
In particular, the Lie algebra of the Lie group $\operatorname{Aut}_G(M)$ of the $G$-equivariant automorphisms of $M$ is identified with the Lie algebra of $P/H$ which corresponds to the factor $\aaa_1 \oplus J \aaa_1$ in the previous decomposition. Note that, by definition of $\operatorname{Aut}_G(M)$, this Lie algebra also identifies with the center $\zzz(\ggg)$ of $\ggg$.

We fix a Riemanian metric $g_0$ with Kähler form $\omega_{0} \in 2 \pi \, c_1(M) $ on M. By the $\partial
\overline{\partial}$-lemma, there is a unique function $h$ in $\C^{\infty}(M,\RR)$ such that
\begin{equation}\label{Dada1}
\Ric(\omega_{0}) - \omega_{0} = \sqrt{-1} \partial
\overline{\partial} h, ~~ \int_M e^h \omega_{0}^n = \int_M \omega_{0}^n.
\end{equation}
In addition, if we fix a Hermitian metric $ m^0 $ on $K_M^{-1} $ such that $ \omega_{m^0} = \omega_{g^0} $, then we can define a volume $ dV $ on $M$ given in a local trivialisation $ s $ of $ K_M^{-1} $ by
$$
dV= \vert s \vert_{m^0} \, s^{-1} \wedge \overline{s^{-1}} = e^{-\varphi} s^{-1} \wedge \overline{s^{-1}},
$$ 
where $\varphi$ is the local potential with respect to the trivialization $s$. Up to an additive constant, $ h $ is the logarithm of the potential of $ dV $ with respect to $ \omega^n_{g^0} $, so we renormalize in order to have
\begin{equation}\label{renormal}
e^{h} \omega_{g^0}^n = dV.
\end{equation}
Indeed, by writing locally Equality \eqref{Dada1}, we get that
$$
\partial \overline{\partial} \left( \ln \det((g_0)_{i \overline{j}}) +h \right) = \partial \overline{\partial} \left( \ln \det(\varphi_{i \overline{j}}) +h \right) = \partial \overline{\partial} \varphi,
$$
where $\varphi$ is the local potential in this open set. This last equation can finally be written globally 
$$
\partial \overline{\partial} \left( \ln e^h \, \cfrac{\omega^n_{g^0}}{dV} \right)=0. 
$$
We conclude by using the maximum principle.

\subsection{Determination of the solitonic vector field}

The first step in order to prove the existence of a Kähler-Ricci soliton is to determine the solitonic vector field. To do this, we use the Futaki invariant. We have the following result (see proposition $2.1$ in \cite{TZ2}):
\begin{proposition}\label{Fnuletar}
There exists a unique complex holomorphic vector field $X \in \eta_r(M)$ with $\operatorname{Im}(X) \in \kkk$ such that the Futaki invariant of $X$, noted $F_X : \eta(M) \rightarrow \CC$, vanishes on $\eta_r(M)$. Moreover, $X$ is equal to zero or belongs to the center of $\eta_r(M)$, and we have
$$
\forall (u,v) \in \eta_r(M) \times \eta(M),~~F_X([u,v])=0.
$$
In particular $F_X(\cdot)$ is a Lie character on $\eta_r(M)$.
\end{proposition}
By applying this theorem to the horospherical case, we obtain the following result
\begin{proposition}\label{formX}
The vector field $ X \in \eta_r(M)=\ggg^{1,0}$ given by proposition \ref{Fnuletar} has the following form:
$$
X = \xi - \sqrt{-1} J \,\xi \text{ where } \xi \in \aaa_1.
$$
\end{proposition}

\begin{proof}
See proposition 3.2 of \cite{Delgove}
\end{proof}

\section{Convergence of the Kähler-Ricci Flow}

\subsection{Preliminaries}
Let $M$ a connected compact Kähler manifold. A familly $(\omega_t)_{t \in I}$ of Kähler form on $M$ defined on  a interval $I$ of $\RR$ containing$0$ is called a \textit{solution of the Kähler-Ricci flow with initial condition $\omega_0$} if 
\begin{equation}\label{KRF}
\cfrac{\partial}{ \partial t} \, \omega_t = -Ric(\omega_t),~~~~ \omega_t \vert_{t =0} = \omega_0.
\end{equation}
If we assume that $M$ is a Fano manifold i.e. $c_1(M)>0$ then we renormalize the Kähler-Ricci flow by the change of variables $ \tilde{\omega}_t = e^t \, \omega_{1-e^{-t}}$ and we obtain \textit{the renormalized Kähler-Ricci flow} :
\begin{equation}\label{NKRF}
\cfrac{\partial}{ \partial t} \, \tilde{\omega}_t = -Ric(\tilde{\omega}_t) + \tilde{\omega}_t,~~~~ \tilde{\omega}_t \vert_{t =0} = \omega_0.
\end{equation}
This renormalization enables the flow to verify interesting properties 
\begin{lemma}
Let $M$ a connected Fano compact Kähler manifold. If $(\omega_t)_{ t \in [0,T[}$ a solution of the renormalized Kähler-Ricci flow \eqref{NKRF} with initial value $\omega_0 \in 2 \pi \, c_1(M)$ on a interval $[0,T[$ where $T>0$ then $\omega_t \in 2 \pi \, c_1(M)$ for all $t \in [0,T[$. Moreover, we have
\begin{itemize}
\item[$\bullet$]for all $t \in [0,T[$, \begin{equation}\label{Vol}
\operatorname{Vol}(M,\omega_t)= \operatorname{Vol}(M, \omega_0),
\end{equation}
\item[$\bullet$] for all $t \in [0,T[$, there exists a function $\varphi_t \in \C^{\infty}(M,\RR)$ such that
\begin{equation}\label{existpop}
\omega_t =\omega_0 + \sqrt{-1} \, \partial \overline{\partial}  \, \varphi_t.
\end{equation}
\end{itemize}
\end{lemma}
\begin{proof}
By taking the cohomology class of Equation \eqref{NKRF}, we obtain the ordinary first order differential equation
\begin{align*}
\cfrac{\partial}{ \partial t}~ [ {\omega_t}]  = -[Ric(\omega_t)] + [\omega_t] = - 2 \pi \, c_1(M) + [\omega_t],
\end{align*}
with the intial value $[\omega(0)]= 2 \pi \, c_1(M)$ so solving this equation gives us $[\omega(t)]= 2 \pi \, c_1(M)$.
The first point is hence a consequence of the Stokes theorem and the last point is a consequence of the $ \partial \overline{\partial}$-lemma.
\end{proof}

Let $g_0$, $\omega_0$ and $\eta(M)$ be as in section \ref{setuphoro}. By the Hodge theory, for all $X \in \eta(M)$ there exists a unique function $\theta_X \in \C^\infty(M,\CC)$ satisfies 
$$
i_X \omega_0 = \sqrt{-1} \,\overline{\partial} \theta_X,
~~
\int_M e^{\theta_X} \omega^{n}_0 = \int_M \omega^{n}_0.
$$
and so, by the Cartan formula,
\begin{equation}\label{thetaXLX}
\LLL_X \omega_0 = \sqrt{-1} \partial \overline{\partial} \theta_X.
\end{equation}
\begin{lemma}[Proof of Proposition 2.1 in \cite{TZ2}]
Let $\varphi \in \C^\infty(M,\RR)$ such that $\omega_\varphi := \omega_0 + \sqrt{-1} \partial \overline{\partial} \varphi$ is a Kähler form. The function $\theta_X(\omega_\varphi)$ associated with the vector field $X \in \eta(M)$ for the metric $\omega_\varphi$ satisfies 
\begin{equation}
\theta_X(\omega_\varphi)= \theta_X + X(\varphi), 
\end{equation}
i.e. we get
$$
\LLL_X \omega_\varphi = \sqrt{-1} \partial \overline{\partial} \left( \theta_X + X(\varphi) \right), \int_M e^{\theta_X +X(\varphi)} \omega^{n}_\varphi = \int_M \omega^{n}_\varphi.
$$
\end{lemma}
For the rest of the paper, in order to simplify the equations, we write $\theta_X$ for $\theta_X(\omega_0)$ and we will precise the metric if the situation requires it. Moreover, we have the well-known following result (for example \cite{IKRF} for a proof).
\begin{proposition}\label{exittinfity}
Let $(M,\omega_0)$ be a connected Fano compact Kähler manifold. A familly $(\omega_t)_{t \in [0,T[}$ is a solution of the renormalized Kähler-Ricci flow \eqref{NKRF} if and only if there exists a familly $(\varphi_t)_{t \in [0,T[}$ of smooth functions on $M$ satisfies
\begin{equation}\label{NKRFpop}
\cfrac{\partial \varphi_t}{\partial t} = \log \det((g_0)_{i \overline{j}} + (\varphi_t)_{i \overline{j}}) - \log \det((g_0)_{i \overline{j}}) + \varphi_t - h, ~~\varphi_0 =0, \omega_0 + \sqrt{-1} \, \partial \overline{\partial} \, \varphi_t>0.
\end{equation}
such that
$$
\omega_t = \omega_0 + \sqrt{-1} \, \partial \overline{\partial} \, \varphi_t.
$$
Moreover, the solution of \eqref{NKRF} (and equivalently for \eqref{NKRFpop}) exists for $t \in [0,+ \infty[$ and are unique.
\end{proposition}
We observe that $\frac{\partial \varphi_t}{\partial t} \vert_{t=0} = h$. Let $h_t : M \rightarrow \CC$ be defined by Equation \eqref{Dada1} with $\omega_0$ replaced by $\omega_t$.
Now, we recall a deep estimate due to Perelman. We can consult \cite{TZx,SH} for a proof.
\begin{lemma}\label{Perelman}
Let $(\varphi_t)_{t \in [0,+\infty[}$ a familly of solutions of the equation \eqref{NKRFpop}. By assuming the constant $c_t$ such that $h_t = - \frac{\partial \varphi_t}{\partial t} +c_t$ satisfies
$$
\int_M e^{h_t} \omega_{\varphi_t}^n = \int_M \omega^n_0,
$$
there exists a constant $A$ independent of the time $t$
such that
$$
\vert h_t \vert \leq A.
$$
\end{lemma}
Now, we assume that $M$ is as in Section \ref{setuphoro} and that $\omega_0$ is a $K$-invariant Kähler form. If we denote by $m_0$ a $K$-invariant Hermitian form on $K_M^{-1}$ such that $\omega_0=\omega_{m^0}$, then there exists a convex potentials $u_0 \in \C^\infty(\aaa_1,\RR)$ (see Equation \eqref{potudef}) i.e we have
$$
\omega_0 \vert_{P/H} = \sqrt{-1} \, \partial \overline{\partial_t} u_0.
$$
Moreover, because the Kähler-Ricci flow perserves the $K$-invariance, we obtain that $\omega_t$ admit also a convex potentials $u_t  \in \C^\infty(\aaa_1,\RR)$ which satisfies
$$
\forall x \in \aaa_1,~~ u_t(x)=u_0(x) + \varphi \vert_{P/H} (\exp(x)H).
$$
Hence, by $K$-invariance and using Theorem $\ref{delomega}$ (in particular, Equation \eqref{omegansr}) and the normalisation \eqref{renormal}, we can reduce Equation \eqref{NKRF} to the following real Monge-Ampère equation for the familly of convex potential $(u=u_t)_{t \in [0,+\infty[}$ belonging to $\C^\infty(\aaa_1,\RR)$ : 
\begin{equation}\label{KRF-MAER}
\cfrac{\partial u}{ \partial t} = \log \det(u_{ij}) + u + \sum_{\alpha \in \Phi_P^+} \log \langle \alpha, \nabla u + 4 \tau_{\rho_P} \rangle, ~~~~ u \vert_{t=0}=u_0.
\end{equation}
Thanks to Proposition \ref{exittinfity}, there exists a family $(u_t)$ solution of \eqref{KRF-MAER} for all $t \in [0,+\infty[$.

\subsection{Study of the solutions $(u_t)_{t \in [0,+\infty[}$ of Equation \eqref{KRF-MAER}}

We recall a useful lemma.

\begin{lemma}[\cite{Mi}]\label{ellip}
Let  $\Omega$ be a convex bounded domain of $\RR^n$. Then there is a unique ellipsoid $E$, called the minimal ellipsoid of $\Omega$, whose volume is minimal among ellipsoids containing $\Omega$. In addition, $E$ satisfies
$$
\cfrac{1}{n} E \subset \Omega \subset E.
$$
Let $T$ be an affine transformation preserving the center $x_0$ of $E$ i.e. $T(x)=T_{vect}(x-x_0)+x_0$ for a matrix $T_{vect} \in \operatorname{SL_n}(\RR)$ and such that $T(E)$ is a ball $B_R$ with center $x_0$ for a certain $R>0$ (depending on $E$). In particular, we have $B_{R/n} \subset T(\Omega) \subset B_R$.
\end{lemma}

We set $\overline{u}_t = u_t - c_t$ where $c_t$ is the constant defined in Lemma \ref{Perelman}.

\begin{lemma}\label{wtext}
The function $\overline{u}_t$ attains its minimum $m_t$ at a point $x_t \in \aaa_1$ .
\end{lemma}
\begin{proof}
Recall that convex function on $ \aaa_1 $ that admits a critical point admits a global minimum. Moreover, we can work with the function $u_t$ because it differs with $\overline{u}_t$ by a constant. Note that $ u_t $ is a convex function thanks to Proposition \ref{potconv}. In order to conclude, it is therefore sufficient to prove that $ 0 \in \nabla u_t (\aaa_1) $ and it this follow from Equation \eqref{0inD} and Proposition \eqref{delomega}.
\end{proof}

\begin{lemma}\label{mt<C}
We have
$$
\exists C>0, \forall t \in [0,+ \infty], m_t \leq C.
$$
\end{lemma}
\begin{proof}
First we show the lower bound. By Lemma \ref{Perelman}, one have
$$
\vert \cfrac{\partial \varphi_t}{\partial t} -c_t \vert \leq A,
$$
so, because $\varphi_t \vert_{P/H}( \exp(x) H) = u_t(x)$, we get
\begin{equation}\label{Anni}
-A \leq \cfrac{\partial u_t}{\partial t} -c_t \leq A.
\end{equation}
Using Equation \ref{KRF-MAER}, we obtain
$$
-A \leq \log \det(u_{ij}) +u -c_t + \sum_{\alpha \in \Phi_P^+} \log \langle \alpha, \nabla u(x) + 4 \tau_{\rho_P} \rangle \leq A.
$$
and finally we get
$$
e^{-A } \det(u_{ij})  \prod_{\alpha \in \Phi^+_P} \langle \alpha , \nabla u(x) + 4 \tau_{\rho_P} \rangle \leq e^{- \overline{u}_t} \leq e^{A} \det(u_{ij}) \prod_{\alpha \in \Phi^+_P} \langle \alpha , \nabla u(x) + 4 \tau_{\rho_P} \rangle.
$$
We integrate the previous equation in order to get
$$
e^{-A} \int_{\aaa_1} \prod_{\alpha \in \Phi^+_P} \langle \alpha , \nabla u(x) + 4 \tau_{\rho_P} \rangle \det(u_{ij})(x)dx \leq \int_{\aaa_1} e^{-\overline{u}} dx \leq e^{A} \int_{\aaa_1} \prod_{\alpha \in \Phi^+_P} \langle \alpha , \nabla u(x) + 4 \tau_{\rho_P} \rangle \det(u_{ij})(x)dx, 
$$
and by using Proposition \ref{potconv}, we finally have
\begin{equation}\label{e^uborn}
e^{-A}  \, \int_{-2 \Delta^+} \prod_{\alpha \in \Phi^+_P} ( \alpha , p  ) dp \leq \int_{\aaa_1} e^{-\overline{u}_t} dx \leq e^{A} \, \int_{-2 \Delta^+} \prod_{\alpha \in \Phi^+_P} ( \alpha , p  ) dp. 
\end{equation}
Now we recall, thanks to Proposition \ref{potconv} that $\vert \nabla \overline{u}_t(x) \vert \leq d_0:= \sup_{x \in 2 \Delta} \Vert x \Vert $and so, by using the mean value theorem,
$$
-d_0 \Vert x - x_t \Vert \leq m_t - \overline{u}_t(x) \leq d_0 \Vert x - x_t \Vert .
$$
By Equation \eqref{e^uborn}, there exists a time independent constant $c_2$ such that
$$
\int_{\aaa_1} e^{-d_0 \Vert x - x_t \Vert } dx \leq  \int_{\aaa_1} e^{-\overline{u}_t + m_t} dx \leq c_2 e^{m_t}.
$$
By doing the change of variable $y=x-x_t$, the left term is independent of  $t$ and 
$$
\int_{\aaa_1} e^{-d_0 \vert y \vert } dy \leq c_2 e^{m_t}.
$$
Hence we can conclude that there exists a constant $C>0$ such that for every $t \in [0, + \infty[$,
$$
C \leq m_t.
$$
Now, we show the upper bound. The proof is inspired by the proof of Lemma 2.1 in \cite{WZ}. We set
$$
A_k := \lbrace x \in \aaa_1 ~/~ m_t + k \leq \overline{u}_t(x) \leq m_t + k +1 \rbrace.
$$
Since $\overline{u}$ is convex and $m_t < + \infty$, we have
\begin{itemize}
\item[$\bullet$] $A_k$ is bounded for all $k \geq 0$ and $\cup_{k \in \NN} A_k = \aaa_1$
\item[$\bullet$] $x_t \in A_0$.
\item[$\bullet$] $\cup_{i=0}^k A_i$ is convex for all $k \geq 0$.
\end{itemize}
Moreover, there exists a constant $c>0$ independent of the time such that
$$
\det( (\overline{u}_t)_{ij}) \geq   c  e^{- \overline{u}_t} \text{ in $\aaa_1$.}
$$
Indeed, thanks to Equation \eqref{KRF-MAER}, we have
\begin{align*}
\det( (\overline{u}_t)_{ij}) &=   \exp \left[ \cfrac{\partial u}{\partial t} - u \right] \prod_{\alpha \in \Phi^+_{P}} \langle \alpha, \nabla u(x) + 4 \tau_{\rho_P} \rangle^{-1} \\
&=  \exp \left[ \cfrac{\partial u}{\partial t} -c_t - \overline{u} \right]  \prod_{\alpha \in \Phi^+_P} \langle \alpha, \nabla u(x) + 4 \tau_{\rho_P} \rangle^{-1}. \\
& \geq e^{-A} \, f^{-1} \, e^{-\overline{u}_t}, \\ 
\end{align*}
by the Equation \eqref{Anni}, where
$$
f:= \sup_{ p \in -2 \Delta^+} \left\lbrace \prod_{\alpha \in \Phi^+_P} ( \alpha , p ) \right\rbrace>0.
$$
Hence, we get
\begin{equation}\label{Wagner}
\det( (\overline{u}_t)_{ij}) \geq  C_0 \, e^{- m_t} \text{ dans $A_0$.}
\end{equation}
By using Lemma \ref{ellip}, there exists a affine transformation $ y = T (x) $ such that the vectorial part has a determinant egal to $1$ and preserving the center of the minimal ellipsoid of $A_0 $ and satisfying
\begin{equation}\label{Wagner2}
B_{R/r} \subset T(A_0) \subset B_R,
\end{equation}
and 
\begin{equation}\label{cloclo}
\det ((\overline{u}_t)_{ij}) \geq C_0  e^{-m_t}  \text{ in $T(A_0)$.}
\end{equation}
Moreover, we clain 
\begin{equation}\label{Anni2}
R \leq \sqrt{2} \, r \, C_0^{-1/2r} e^{m_t/2r}.
\end{equation}
Indeed, let
$$
v : y \in \aaa_1   \longmapsto   \cfrac{1}{2} C_0^{-1/r} \, e^{m_t/r} \left[ \Vert y - y_t \Vert^2 - \left(\cfrac{R}{r} \right)^2 \right] + m_t +1 \in \RR
$$
where $y_t$ is the center of the minimal ellipsoid of $A_0$. A direct computation gives us
$$
\det(\overline{u}_{ij})= C_0 e^{-m_t} ~~ \text{ on $T(A_0)$,}
$$
and $v \geq \nu$ on $ \partial T(A_0)$ hence on $T(A_0)$ thanks to the comparison principle. In particular, we get
$$
m_t  \leq \overline{u}_t \leq v(y_t) = - \cfrac{1}{2} \, C_0^{1/r} \, e^{m_t/r} \left(\cfrac{R}{r} \right)^2 + m_t+1.
$$
Moreover, by the convexity of $\overline{u}_t$, we obtain 
$$
A_k \subset \bigcup_{i=0}^k A_k \subset (k+1) \cdot A_0,
$$
where $ (k+1) \cdot A_0$ is the dilation of $A_0$ of factor $(k+1)$. Moreover, thanks to the equation \eqref{Wagner2}, we get
$$
T(A_k) \subset  T( ((k+1) \cdot A_0)) \subset (k+1) \cdot T(A_0) = B_{(k+1)R}.
$$
Now, if we denote by $\omega_{r}$ the volume of the unit ball of $\aaa_1$ then we have by Equation \eqref{Anni2} 
\begin{align*}
\int_{\aaa_1} e^{-\overline{u}_t} & \leq \sum_k \int_{T(A_k)} e^{-\overline{u}_t}  \leq \sum_k e^{-m_t -k} \vert T(A_k) \vert 
\leq  \omega_{r} \sum_k e^{-m_t -k} \vert (k+1)R \vert^{r}  \\
&= \omega_{r}  \cfrac{(R)^{r}}{e^{m_t}} \sum_k \cfrac{(k+1)^{r}}{e^k} 
 \leq B e^{-m_t/2}.
\end{align*}
where $B>0$ is a constant independent of $t$. Finally we get, by using Equations \eqref{KRF-MAER} and \eqref{Anni} and Proposition \ref{potconv},
\begin{align*}
e^{-m_t/2} & \geq \cfrac{1}{B} \int_{\aaa_1} e^{-\overline{u}_t}  \geq \cfrac{1}{B}  \int_{\aaa_1} \det(u_{ij}) \exp \left[ -\cfrac{  \partial u}{\partial t}  + c_t \right] \prod_{\alpha \in \Phi^+_P} \langle \alpha , \nabla u_t + 4 \tau_{\rho_P} \rangle dx \\
& \geq \cfrac{e^{-A}}{B}  \int_{-2 \Delta^+}  \prod_{\alpha \in \Phi^+_P} ( \alpha , p ) dp =: \tilde{C}, \\
\end{align*}
where $\tilde{C}$ is a constant independent of $t$. Finally, we get
$$
m_t \leq C,
$$
where $C>0$ is a constant independent of $t$.
\end{proof}
Now, we set
\begin{equation}\label{ubb}
 \overline{\overline{u}}_t (\cdot):= \overline{u}_t( \cdot + x_t) -m_t,
\end{equation}
and
\begin{equation}\label{phibar}
\overline{\varphi}_t:= \overline{\overline{u}}-u_0.
\end{equation}
Remark that $\overline{\varphi}_t$ extends to a smooth function on $M$ which is the potential of $\exp(x_t)^*\omega_t$ with respect to the Kähler form $\omega_0$, that so 
$$
\exp(x_t)^* \omega_t - \omega_0 = \sqrt{-1} \, \overline{\partial} \partial \, \overline{\varphi}_t.
$$

\begin{lemma}\label{supphi<C}
There exists a constant $C'>0$ independent of $t$ such that for every $t \in [0,+\infty[$
$$
\vert \sup_M \overline{\varphi}_t \vert \leq C'.
$$
\end{lemma}
\begin{proof}
By $K$-invariance and density of $G/H$, it is sufficient to prove the result on $P/H$ i.e. where the function $\overline{\varphi}_t$ has the following expression
\begin{align*}
\overline{\varphi}_t &= \overline{u}_t( \cdot + x_t) - m_t - u_0
= u_t( \cdot + x_t) - m_t - c_t -u_0.
\end{align*}
By convexity, we have
$$
m_t=\overline{u}_t(0) \geq \overline{u}_t(x) + (\nabla \overline{u}_t(x), -x ),
$$
and so, by definition of the support function $v_{2 \Delta}$ and by Proposition \eqref{potconv},
$$
m_t + v(x) \geq \overline{u}_t(x).
$$
Hence, by Proposition \ref{mt<C}, there exists $F$ independant of $t$ such that
\begin{align*}
\overline{\varphi}_t (x) & \leq v(x) - u_0(x) + m_t  \leq F .
\end{align*}

For the lower bound, using Equation \eqref{KRF-MAER}, the function $\overline{\overline{u}}_t$ satisfies the following real Monge-Ampère equation on $\aaa_1$ :
$$
\det((\overline{\overline{u}}_t)_{ij})= \exp \left[ \cfrac{\partial u}{\partial t}( \cdot + x_t) -c_t - \overline{\overline{u}}_t -m_t \right] \prod_{\alpha \in\Phi_P^+} \langle \alpha, \nabla \overline{ \overline{u}}_t + 4 \tau_{\rho_P} \rangle^{-1}.
$$
By density since $\overline{\overline{u}}_{t} = \overline{\varphi}_t +u_0$ and $u_t=\varphi_t + u_0$, we see that $\overline{\varphi}_t$ satisfies the following complex Monge-Ampère equation on $M$:
$$
\det((g_0)_{i \overline{j}} + (\overline{\varphi}_t)_{i \overline{j}}) = \exp \left[h + \cfrac{\partial \varphi_t }{\partial t}( \cdot + x_t) - c_t - \overline{\varphi}_t -m_t \right].
$$
Now, since $h$ is a continuous function on a compact manifold and thanks to Lemma \ref{Perelman}, there exists a constant $B'>0$ independent of $t$ such that
$$
\det((g_0)_{i \overline{j}} + (\overline{\varphi_t})_{i \overline{j}})e^{-B'} \leq e^{- \overline{\varphi}_t -m_t} \leq (g_0)_{i \overline{j}} + (\overline{\varphi_t})_{i \overline{j}})e^{B'},
$$
and hence, by integrating,
$$
e^{-B'} \int_M \det( (g_0)_{i \overline{j}} + (\overline{\varphi_t})_{i \overline{j}})) dV \leq \int_M e^{- \overline{\varphi}_t -m_t} dV \leq e^{B'} \int_M \det((g_0)_{i \overline{j}} + (\overline{\varphi_t})_{i \overline{j}}) dV. \\
$$
By density we can reduce the integration on $G/H$ and doing an integration along the fiber, there exits a constant $C''$ independent of $t$ such that
\begin{align*}
e^{-B'} \, C'' \int_{\aaa_1} \det( (\overline{\overline{u}}_t)_{ij}) \prod_{\alpha \in \Phi_P^+} \langle \alpha, \nabla \overline{\overline{u}}_t(x) + 4 \tau_{\rho_P} \rangle dx & \leq \int_M e^{- \overline{\varphi}_t -m_t} dV \\ & \leq e^{B'} \, C'' \int_{\aaa_1} \det((\overline{\overline{u}}_t)_{ij}) \prod_{\alpha \in \Phi_P^+} \langle \alpha, \nabla \overline{\overline{u}}_t(x) + 4 \tau_{\rho_P} \rangle dx
\end{align*}
So, by doing the change of variables $y= \nabla \overline{\overline{u}_t}(x) + 4 \tau_{\rho_P}$ and applying Proposition \ref{potconv}, there exist two constants time independent $c_1$ and $c_2$ such that
\begin{equation}\label{borneinequation}
c_1:= e^{-B'} \, C'' \int_{-2 \Delta^+} \prod_{\alpha \in \Phi_P^+} (\alpha, p ) dp   \leq \int_M e^{- \overline{\varphi}_t -m_t} dV \leq c_2:= e^{B'} \, C''\int_{-2 \Delta^+} \prod_{\alpha \in \Phi_P^+}( \alpha, p ) dp .
\end{equation}
Now, since $M$ is compact, this implies that there exists a time independent constant $C_0'$ such that
\begin{equation}\label{supinfbarvarphi}
\sup_M (\overline{\varphi}_t + m_t ) \geq -C_0 \text{ et } \inf_M (\overline{\varphi}_t + m_t) \leq C_0.
\end{equation}
We can conclude, thanks to Lemma \ref{mt<C}, as wanted
$$
\sup_M \overline{\varphi} \geq -C_0' - C.
$$
\end{proof}
Let us end this section with the following lemma.

\begin{lemma}\label{majosc}
There exists a constant $C_0>0$ such that
$$
osc_M \overline{\varphi}_t := \sup_M \overline{\varphi}_t -\inf_M \overline{\varphi}_t \leq C_0 \left( 1 + I(\overline{\varphi})\right)^{n+1},
$$
where
$$
I(\overline{\varphi}_t) = \frac{1}{V} \int_M \overline{\varphi}_t (\omega^n_0 - \omega^n_{\overline{\varphi}}),
$$
where we denote 
$$
V= \operatorname{Vol}(M,g_0).
$$
\end{lemma}
\begin{proof}
The proof is stated in \cite{ZZZ} and inspired at the origin by Proposition 3.1 of \cite{TZx}.
\end{proof}
\subsection{$\C^0$-estimate of $\overline{\varphi}_t$}

This section recalls general results which are extracted to \cite{ZZZ}. In particular, we recall (and modify when necessary) Propositions 3.1 and 3.2 and Corollary 3.1 

Recall (see Proposition \ref{formX}) that we have a vector field $X \in \zzz(\ggg)$ such that 
\begin{equation}\label{Wagner3}
F_{X^{1,0}} \equiv 0.
\end{equation}
Now, we consider the one paremeter subgroup of automorphims $\sigma_t = \exp(t \, X)$ generated by $X$. We define the familly of potentials $(\varphi_t')_{t \in [0,+\infty[}$ by the formula :
$$
\sigma^*_t \omega_{\varphi_t} = \omega_0 + \cfrac{\sqrt{-1}}{2 \pi} \partial \overline{\partial} \varphi_t '.
$$
Now, thanks to the equation \eqref{KRF-MAER}, we have
$$
\cfrac{\partial  \omega_{\varphi_t'}}{\partial t} = -Ric(\omega_{\varphi_t'}) + \omega_{\varphi_t'} + \LLL_X \omega_{\varphi_t'}.
$$
Moreover, by the maximum principle, we show that the previous equation is equivalent modulo a constant to the following equation
\begin{equation}\label{equationphi'}
\cfrac{\partial \varphi_t'}{\partial t}  =  \log \omega^n_{\varphi_t'} - \log \omega^n_0 + \varphi_t' +\theta_X + X(\varphi_t') -h. 
\end{equation}
Now, we introduce the following space of potentials 
$$
\mathcal{M}_X(\omega_0):= \lbrace \varphi \in  \C^{\infty}(M) ~/~ \omega_{\varphi} := \omega_0 + \sqrt{-1} \partial \overline{\partial} \varphi >0,~~ \text{Im}(X)(\varphi)=0 \rbrace.
$$
and we define the \textit{generalized K$-$energy functional associated to $X$} defined in $\mathcal{M}_X(\omega_0)$ by
\begin{align*}
\tilde{\mu}(\varphi) = & - \cfrac{n \sqrt{-1}}{2 \pi V} \int_0^1 \int_M \dot \psi_t \left[ Ric(\omega_{\psi_t}) - \omega_{\psi_t}  \right] \wedge e^{\theta_X(\omega_{\psi_t})} \omega^{n-1}_{\psi_t} \wedge dt \\
& + \cfrac{n} {4 \pi^2 V} \int_0^1 \int_M  \dot \psi_t  \left[ \partial \overline{\partial} \theta_X (\omega_{\psi_t}) \right] \wedge e^{\theta_X(\omega_{\psi_t})} \omega^{n-1}_{\psi_t} \wedge dt \\
& + \cfrac{n} {4 \pi^2 V} \int_0^1 \int_M \dot \psi_t  \left[ \partial \left( h_{\omega_{\psi_t}} - \theta_X (\omega_{\psi_t}) \right) \wedge \overline{\partial} \theta_X(\omega_{\psi_t}) \right]  \wedge e^{\theta_X(\omega_{\psi_t}} \omega^{n-1}_{\psi_t} \wedge dt, 
\end{align*}
where $(\psi_t)_{t \in [0,1]}$ is a path connecting 0 to $\varphi$ in $\mathcal{M}_X(\omega_0)$ and $\dot \psi_t= \frac{\partial \psi_t}{\partial t}$. 

\begin{lemma}\label{muphi}
Let $\sigma \in Aut_r(M)$. If we define the potential $\varphi_{\sigma_t}$ by
$$
\sigma^* \omega_{\varphi_t} = \omega_g + \sqrt{-1} \partial \overline{\partial} \varphi_{\sigma_t}.
$$
We have $\varphi_{\sigma_t} \in \mathcal{M}_X(\omega_0)$ and $\tilde{\mu}(\varphi)=\tilde{\mu}(\varphi_{\sigma_t})$.
\end{lemma}
\begin{proof}
Because $X \in \zzz(\ggg)$, we have
$$
\forall \rho \in \operatorname{Aut}_r(M),~~\rho^* \LLL_X \omega_{\varphi_t} = \LLL_X (\rho^* \omega_{\varphi_t}),
$$
and so
$$
\varphi_{\rho_t} \in \mathcal{M}_X(\omega_0),~~ \forall \rho \in Aut_r(M). 
$$
Now, if $Y \in \ggg$ then $Y$ induce a one-parameter subgroup in $\operatorname{Aut}_r(M)$ and so, by the above, $\rho_t \in \mathcal{M}_X(\omega_0)$.  Moreover, we can prove (see Computation $2.13$ in \cite{test}) that
$$
\cfrac{d \tilde{\mu}(\varphi_{\rho_t})}{dt}= \operatorname{Re}(F_{X^{1,0}}(Y^{1,0})).
$$
Moreover, we have $F_X(Y)=0$ (see Equation \eqref{Wagner3}) and so $\frac{d \tilde{\mu}(\varphi_{\rho_t})}{dt}=0$ and hence
$$
\tilde{\mu}(\varphi_t)=\tilde{\mu}(\varphi_{\rho_t}).
$$
\end{proof}

\begin{lemma}\label{muphi0}
For all $t \in [0, +\infty[$, we have
$$
\tilde{\mu}(\overline{\varphi}_t) \leq 0.
$$
\end{lemma}
\begin{proof}
First, $\varphi_t' \in \mathcal{M}_X(\omega_0)$ (see Lemma \ref{muphi}) and so we can compute $\tilde{\mu}(\varphi_t')$.
Nowe we get, thanks to Equation \eqref{equationphi'} and doing a integration by part :
\begin{equation}\label{derivemut}
\cfrac{d\tilde{\mu}(\varphi_t')}{d t} = \dfrac{ \sqrt{-1}}{V} \int_M \Vert \overline{\partial} \cfrac{ \partial \varphi_t'}{\partial t} \Vert^2_{\omega_{\varphi_t'}} e^{\theta_X + X(\varphi_t')} \omega^n_{\varphi_t'} \leq 0 .
\end{equation}
Hence, we get
\begin{equation}\label{mu'<0}
\tilde{\mu}(\varphi_t') \leq \tilde{\mu}(\varphi_0') = \tilde{\mu}(0)= 0.
\end{equation}
We can conclude, thanks to Lemma \ref{muphi} and Equation \eqref{mu'<0},
\begin{align*}
\tilde{\mu}(\overline{\varphi})& =\tilde{\mu} (\overline{\varphi}_t + m_t) 
= \tilde{\mu}(\varphi_t)
= \tilde{\mu}(\varphi_t')
 \leq 0.
\end{align*}
\end{proof}
With theses lemmas, we can prove the following $\C^0$-estimate for $\overline{\varphi}_t$.
\begin{proposition}\label{estimeC^0}
There exists a constant $C>0$ independant of time $t$ such that
$$
\Vert \overline{\varphi_t} \Vert_{\C^0(M)} \leq C.
$$
\end{proposition}
\begin{proof}
Firstly, we recall (see \cite{TZ2} the definition functional $\tilde{F}$ on $\mathcal{M}_X(\omega_0)$):
$$
\tilde{F}(\varphi) := \tilde{J}(\varphi) - \cfrac{1}{V} \int_M \varphi e^{\theta_X} \omega^n_0 - \log \left( \dfrac{1}{V} \int_M e^{h-\varphi} \omega_0^n \right),
$$
with $\tilde{J}$ is defined by:
$$
\tilde{J}(\varphi):= \dfrac{1}{V} \int_0^1 \int_M
 \dot \psi_t  \left( e^{\theta_X}  \right) \omega^n_0 \wedge dt - \dfrac{1}{V} \int_0^1 \int_M
 \dot \psi_t  \left( e^{\theta_X + X(\psi_t)}  \right) \omega^n_{\psi_t} \wedge dt.  
 $$
where $\\psi_t$ is a path connecting $0$ to $\varphi$  in $\mathcal{M}_X(\omega_g)$ and $\dot \psi_t$ is the derivative of $\psi_t$ with respect to $t$. We can show that $\tilde{J}$ is independant of the path $\psi$ and that $\tilde{J}(\varphi)>0 $ (voir \cite{WZ}). We also prove ( see Lemma 5.1 of \cite{TZ2} or Formula 67 of \cite{DarvasRUbi}) that there exits a constant $C>0$ such that :
$$
\tilde{\mu}(\varphi) \geq \tilde{F}(\varphi) - C,~~\forall \varphi \in \mathcal{M}_X(\omega_0).
$$
By using Lemma \ref{muphi0} and the definition of $\overline{\varphi}$, we get
\begin{equation}\label{Fborn}
\tilde{F}(\overline{\varphi_t}+ m_t)= \tilde{F}(\varphi_t) < C,
\end{equation}
and hence by using Equations \eqref{Fborn} and \eqref{borneinequation} and Lemmas \ref{mt<C} and \ref{supphi<C}
\begin{align*}
\tilde{J}(\overline{\varphi}_t) &= \tilde{J}(\overline{\varphi}_t+ m_t)  = \tilde{F}(\overline{\varphi}_t + m_t) + \cfrac{1}{V} \int_M \left( \overline{\varphi}_t + m_t \right) e^{\theta_X} \omega^n_0 + \log \left( \dfrac{1}{V} \int_M e^{h- \overline{\varphi}_t - m_t} \omega_0^n \right)   \\
& \leq \tilde{F}(\overline{\varphi}_t+ m_t) + \cfrac{1}{V} \int_M \left( \overline{\varphi}_t + m_t \right) e^{\theta_X} \omega^n_0 + C_1  \leq  \tilde{F}(\overline{\varphi}_t+ m_t) + C_2 + C_1  \\
&  \leq  C + C_2 + C_1=:C_3  \\
\end{align*}
Now, if we set
$$
I(\overline{\varphi}_t) := \frac{1}{V} \int_M \overline{\varphi}_t (\omega^n_0 - \omega^n_{\overline{\varphi}_t}),
$$
then we know that, thanks to \cite{CTZ}, there exists a uniform constant $c>0$ such that
$$
\tilde{J}(\overline{\varphi}_t) \geq c \, I(\overline{\varphi}_t).
$$
By using, the preivous inegality, we get 
$$
I(\overline{\varphi}_t) \leq C_4:=C_3/c.
$$
Now Lemma \ref{majosc} tell us that
$$
osc_M \overline{\varphi}_t \leq C_5.
$$
and,by Equation \eqref{supinfbarvarphi}, 
$$
\Vert \overline{\varphi}_t + m_t \Vert_{\C^{0}(M)} \leq osc_M ( \overline{\varphi}_t +m_t) +2 \tilde{C} = osc_M \overline{\varphi}_t +2 \tilde{C} \leq C_6.
$$
Lemma \ref{mt<C} allows to conclude.
\end{proof}
\subsection{Uniforme estimate of $c_t$ and $\partial \varphi_t / \partial t$}
Now, we prove in this section that modulo a renormalisation of $h_t$, we can get uniform bounds of  $c_t$ and $\partial  \varphi_t / \partial t$. Before, we prove a lemma about $K$-energy functionnal :
\begin{lemma}\label{bornsupmu}
Let $(\varphi_t)_{t \in [0,+\infty[}$ be a solution of Equation \ref{NKRF}. We have
$$
\tilde{\mu}(\varphi_t) \geq C.
$$
\end{lemma}
\begin{proof}
Indeed, we have, thanks to Equation \eqref{Fborn}, Lemma \ref{mt<C}, Proposition \ref{estimeC^0} and because $\tilde{J}(\overline{\varphi})>0$,
\begin{align*}
\tilde{\mu}(\varphi_t) &=  \tilde{\mu}(\overline{\varphi}_t +m_t) \geq \tilde{F}(\overline{\varphi}_t +m_t) -C 
 = \tilde{J}(\overline{\varphi}_t) - \cfrac{1}{V} \int_M \left( \overline{\varphi}_t - m_t \right) e^{\theta_X} \omega^n_0 + \log \left( \dfrac{1}{V} \int_M e^{h- \overline{\varphi}_t - m_t} \omega_0^n \right) -C \\
 & = \tilde{J}(\overline{\varphi}_t) - \cfrac{1}{V} \int_M \left( \overline{\varphi}_t - m_t \right) e^{\theta_X} \omega^n_0 - C_1 -C  = \tilde{J}(\overline{\varphi}_t)  - C_2 - C_1 -C \geq - C_2 - C_1 -C. 
\end{align*}
\end{proof}
There exists a constant $C>0$ independant of time such that
\begin{equation}\label{Theta<C}
\Theta:= \dfrac{1}{V} \int_0^\infty \int_M  \Vert \overline{\partial} \cfrac{\partial \varphi_t'}{\partial t} \Vert^2 e^{\theta_X + X(\varphi_t')} \omega_{\varphi_t'}^n \wedge dt < C.
\end{equation}
Indeed, by using Equation \ref{derivemut}, we have
$$
- \Theta= \int_0^{+ \infty} \cfrac{d \tilde{\mu}(\varphi_t')}{d t} dt = \lim_{t \rightarrow + \infty} \tilde{\mu} (\varphi_t') - \tilde{\mu}(\varphi_0') = \lim_{t \rightarrow + \infty} \tilde{\mu} (\varphi_t') 
$$
To conclude, it is sufficient to remark (see Lemma \ref{muphi}) that
$$
\tilde{\mu}(\varphi_t') = \tilde{\mu}(\varphi_t).
$$
and, thanks to Lemma \ref{bornsupmu} and \ref{muphi0}, we have that  
$$
\forall t \in [0, + \infty[, ~~ C \leq \tilde{\mu}(\varphi'_t) \leq 0,
$$
and, by taking the limit, the result is proved. Moreover, we get immediately
$$
\tilde{\Theta}:= \dfrac{1}{V} \int_0^\infty \int_M  \Vert \overline{\partial} \cfrac{\partial \varphi_t'}{\partial t} \Vert^2 e^{\theta_X + X(\varphi_t')-t} \omega_{\varphi_t'}^n \wedge dt < C.
$$
\begin{lemma}
\label{lim0}
If we renormalise the function $h$ by adding a constant such that
$$
\dfrac{1}{V} \int_M \left( h - \theta_X \right) \omega_0^n e^{\theta_X} = - \tilde{\Theta}
$$
then 
$$
\lim_{t \rightarrow + \infty} \int_M \dfrac{\partial \varphi_t'}{\partial t} e^{\theta_X + X(\varphi_t')} \omega^n_{\varphi_t'}=0.
$$
In particular, there exists a constant $C>0$ independant of time $t$ such that
$$
\vert \int_M \dfrac{\partial \varphi_t'}{\partial t} e^{\theta_X + X(\varphi_t')} \omega^n_{\varphi_t'} \vert \leq C.
$$
\end{lemma}
\begin{proof}
The proof on is inspired by Lemma $4.2$ of \cite{TZx} but  this lemma assume the existence of a Kähler-Ricci soliton. We set
$$
a_t:=  \int_M \dfrac{\partial \varphi_t'}{\partial t} e^{\theta_X + X(\varphi_t')} \omega^n_{\varphi_t'}.
$$
We must prove $\lim_{t \rightarrow + \infty} a_t =0$. First, we have
\begin{align*}
\dfrac{d a_t}{dt} &= a_t - \int_M \Vert \overline{\partial} \cfrac{\partial \varphi_t'}{\partial t} \Vert^2 e^{\theta_X + X(\varphi_t')} \omega^n_{\varphi_t'},
\end{align*}
and hence
\begin{equation}\label{ddteat}
\dfrac{d (e^{-t} a_t)}{dt} = -\int_M e^{-t} \Vert \overline{\partial} \cfrac{\partial \varphi_t'}{\partial t} \Vert^2 e^{\theta_X + X(\varphi_t')} \omega^n_{\varphi_t'}.
\end{equation}
Moreover, thanks to Equation \eqref{equationphi'}, we get
$$
h- \theta_X = - \cfrac{\partial \varphi_t'}{\partial t} \vert_{t=0},
$$
hence we obtain
\begin{equation}\label{a0}
a_0 = - \int_M \left( h - \theta_X \right) \, e^{\theta_X} \, \omega_0^n ,
\end{equation}
and, by using the hypothesis of renormalisation,
$$
a_0 = \int_0^\infty \int_M  \Vert \overline{\partial} \cfrac{\partial \varphi_t'}{\partial t} \Vert^2 e^{\theta_X + X(\varphi_t')-t} \omega_{\varphi_t'}^n \wedge dt. 
$$
Now, we remark, by using Equation \eqref{ddteat},
\begin{align*}
a_0 - e^{-t} \, a_t & = e^{0} \, a_0 - e^{-t} \, a_t  - \int_0^t \dfrac{d  (e^{-s}a_s)}{ds} ds =  \int_0^t \, \int_M e^{-s} \Vert \overline{\partial} \cfrac{\partial \varphi_s'}{\partial s} \Vert^2 e^{\theta_X + X(\varphi_s')} \omega^n_{\varphi_s'} ds \end{align*}
And we obtain, by using Equation \eqref{a0},
\begin{align*}
a_t &= e^{t} \left[ a_0 + \int_0^t \int_M e^{-s} \Vert \overline{\partial} \cfrac{\partial \varphi_s'}{\partial s} \Vert^2 e^{\theta_X + X(\varphi_s')} \omega^n_{\varphi_s'} ds \right] \\
&=e^{t} \left[ \int_0^{+ \infty} \int_M e^{-s} \Vert \overline{\partial} \cfrac{\partial \varphi_s'}{\partial s} \Vert^2 e^{\theta_X + X(\varphi_s')} \omega^n_{\varphi_s'} ds - \int_0^t \int_M e^{-s} \Vert \overline{\partial} \cfrac{\partial \varphi_s'}{\partial s} \Vert^2 e^{\theta_X + X(\varphi_s')} \omega^n_{\varphi_s'} ds \right] \\
&= e^{t} \left[ \int_t^{+ \infty} \int_M e^{-s} \Vert \overline{\partial} \cfrac{\partial \varphi_s'}{\partial s} \Vert^2 e^{\theta_X + X(\varphi_s')} \omega^n_{\varphi_s'} ds \right] = \int_t^{+ \infty} \int_M e^{t-s} \Vert \overline{\partial} \cfrac{\partial \varphi_s'}{\partial s} \Vert^2 e^{\theta_X + X(\varphi_s')} \omega^n_{\varphi_s'} ds. \\
\end{align*}
To conclude, it is sufficient to remark that
$$
0 \leq a_t  \leq \int_t^{+ \infty} \left( \int_M \Vert \overline{\partial} \cfrac{\partial \varphi_s'}{\partial s} \Vert^2 e^{\theta_X + X(\varphi_s')} \omega^n_{\varphi_s'} \right) ds \, \rightarrow 0 \text{ quand } t \rightarrow +\infty,
$$
because the integral $\leq \int_0^{+ \infty}  ( \int_M \Vert \overline{\partial} \cfrac{\partial \varphi_s'}{\partial s} \Vert^2 e^{\theta_X + X(\varphi_s')} \omega^n_{\varphi_s'}) ds$ is convergent (see Equation \ref{Theta<C}).
\end{proof}

\begin{proposition}\label{c_t<C}
If we renormalize the function $h_t$ by adding a constant such that
$$
\dfrac{1}{V} \int_M \left( h - \theta_X \right) \omega_0^n e^{\theta_X} = - \tilde{\Theta}.
$$
then there exists two constant $C>0$ and $\tilde{C}>0$ such that
\begin{equation}
\vert c_t \vert \leq C \text{ et } \vert \frac{\partial \varphi_t}{ \partial t} \vert \leq \tilde{C}
\end{equation}
\end{proposition}
\begin{proof}
If $(\varphi_t')_{t \in [0,+\infty[}$ is a solution of \eqref{equationphi'} then 
$$
\dfrac{\partial \varphi_t'}{\partial t} = \dfrac{\partial \varphi_t}{\partial t} + \theta_X + X(\varphi_t').
$$
Hence, by Lemma \ref{lim0}, there exist a uniform constant $C>0$ such that
$$
\vert \int_M \left( \dfrac{\partial \varphi_t}{\partial t} + \theta_X + X(\varphi_t') \right) e^{\theta_X + X(\varphi_t')} \omega^n_{\varphi_t'}  \vert \leq C.
$$
Moreover, we know that there exists a uniform bound for $ \vert X(\psi) \vert$ for all $\psi \in \mathcal{M}_X(\omega_0)$ (see Corollary 5.3 in \cite{Zh}) and the function $\theta_X$ is bounded so there exists a constant $C_1>0$  independent of time $t$ such that
$$
\vert \int_M - \dfrac{\partial \varphi_t}{\partial t} \, \omega^n_{\varphi_t'}  \vert \leq C_1.
$$
but $h_t = - \partial \varphi_t / \partial t + c_t$ (see Lemma \ref{Perelman}) so
$
\vert \int_M \left( h_t - c_t \right) \, \omega^n_{\varphi_t'}  \vert \leq C_1.
$
Finally, we get
$$
 \vert c_t \vert \, \vert \int_M  \omega_{\varphi_t'}^n \vert \leq C_1 + \vert \int_M h_t \omega_{\varphi_t'}^n \vert.
$$
Now, thanks to Equation \eqref{Vol} and by definition of $\varphi_t'$, we have
$
\operatorname{Vol}(M,\omega_{\varphi_t'})=V,
$
hence, by Lemma \ref{Perelman}, 
$$
\vert c_t \vert \leq \dfrac{C_1}{V} + A,
$$
and
$$
\vert \cfrac{\partial \varphi}{\partial t} \vert \leq \cfrac{C_1}{V} + 2 A.
$$
\end{proof}
\subsection{Modified Kähler-Ricci Flow}
In this section, the points $x_t$ are modified to points $x_t'$ in order to verify larger smoothness assumptions and Equation \eqref{equationphi'} is modified accordingly. Finally we will then study the convergence of the Kähler-Ricci flow thus modified.

We start by proving the following lemma that will allow us to modify the $x_t$ family later.
\begin{lemma}\label{xtdiff}
We have
$$
\exists C>0,~~ \forall (t,t') \in \RR^2 ~~ \vert t - t' \vert \leq 1  \Rightarrow \vert x_t - x_t' \vert \leq C.
$$
\end{lemma}
\begin{proof}
By Propositions \ref{potconv} and \ref{estimeC^0}, we have
\begin{align*}
\vert \overline{\overline{u}}_t - v_{2 \Delta} \vert  \leq \vert \overline{\overline{u}}_t - u_0 \vert  + \vert u_0 - v_{2 \Delta} \vert \leq \vert \overline{\varphi}_t \vert + \vert u_0 - v_{2 \Delta} \vert \leq C_0 
\end{align*}
We get also
$$
\vert \overline{\overline{u}}_{t'} - v_{2 \Delta} \vert  \leq C_0.
$$
Moreover, by using Lemma \ref{mt<C} and Proposition \ref{c_t<C}, we have
\begin{align*}
&\vert \overline{\overline{u}}_t(x - x_t) - \overline{\overline{u}}_t(x - x_{t'}) \vert 
= \vert \overline{u}_t(x - x_t) - \overline{u}_t(x - x_{t'})  -m_t +m_{t'} \vert \\
&= \vert u_t(x)- u_{t'}(x) -c_t + c_{t'} -m_t +m_{t'} \vert 
= \vert \varphi_t(x) - \varphi_{t'}(x) -c_t + c_{t'} -m_t + m_{t'}\vert \\
&\leq \vert \varphi_t (x)- \varphi_{t'}(x) \vert + \vert c_t \vert + \vert c_{t'} \vert + \vert m_t \vert + \vert m_{t'} \vert \leq \vert \varphi_t - \varphi_{t'} \vert(x) + \vert c_t \vert + \vert c_{t'} \vert +2 \, C_1 \\
& \leq \vert \varphi_t (x) - \varphi_{t'}(x) \vert + 2 \, C_2 +2 \, C_1 \leq C_3\vert t -t' \vert + 2 \, C_2 +2 \, C_1\leq C_4:= C_3 + 2 \,C_2 +2 \, C_1 
\end{align*}

By using the previous inegalities, we have
$$
\vert v_{2 \Delta}(x-x_t) -v_{2 \Delta}(x-x_{t'}) \vert \leq C_5:= 2 \, C_0 + C_4.
$$
If we take $x=x_t$ then
$$
\vert v_{2 \Delta}(x_t-x_{t'}) \vert = \vert v_{2 \Delta}(0)  -  v_{2 \Delta}(x_t-x_{t'}) \vert \leq C_5.
$$
This means that
$$
( x_t - x_{t'} , p ) \leq C_5, ~~\forall p \in 2 \Delta.
$$
To conclude, it is enough to remark that $2 \Delta$ contains a ball of the form $B(0, \delta)$ with $ \delta > 0$. This is due to the fact that we have an isomorphism (given by $\nabla u$) between $\aaa_1 \simeq \RR^r$ and $2 \Delta$ (see Proposition \ref{potconv}). Taking as point $p$ the intersection point between the line $\RR \, (x_t - x_{t'})$ and the ball $B(0,\delta)$, we therefore obtain thanks to the case of equality of the Cauchy-Schwarz theorem:
$$
\delta \, \vert x_t - x_{t'} \vert  = \vert  x_t - x_{t'} \vert \,   \vert p \vert  = ( x_t - x_{t'} , p  ) \leq C_5,
$$
hence
$$
\vert x_t - x_{t'} \vert \leq C:= C_5/\delta.
$$
\end{proof}
Now, we want to build a family $x_{t}'$ for $t \in [0,+\infty[$ such that there exists two constants $D$ et $\tilde{D}$ independent of time $t$ satisfying 
\begin{equation}\label{condxt'}
\vert x_t - x_{t}' \vert \leq D \text{  and  } \vert \dfrac{d x_{t}'}{dt} \vert \leq \tilde{D}.
\end{equation}
In order to realize it, we consider the following family
\begin{equation*}
  \left\{
      \begin{aligned}
       x_i'&:= x_i & \text{si $t= i \in \NN $}\\
       x_t'&:= (1- \epsilon ) \, x_i + \epsilon \, x_{i+1} &  \text{si $t=i + \epsilon $ avec $i \in \NN$ et $\epsilon \in ]0,1[$}.  
      \end{aligned}
    \right.
\end{equation*}
We remark
\begin{align*}
\vert x_t - x_{t}' \vert = 0 \text{ si } t \in \NN,
\end{align*}
and if $t$ is written $t=i + \epsilon $ with $i \in \NN$ and $\epsilon \in ]0,1[$ then, by using Lemma \ref{xtdiff},
\begin{align*}
\vert x_t - x_{t}' \vert &= \vert x_t - x_i + x_i -x_{t}' \vert  \leq \vert x_t - x_i \vert + \vert  x_i -x_{t}' \vert \\
& \leq C +   \vert  x_i -x_{t}' \vert  \leq C +   \epsilon \vert  x_i -x_{i+1} \vert \leq 2 \, C =: D.  \\
\end{align*}
Moreover, we have
$$
\vert \dfrac{d x_{t}'}{dt} \vert \leq \tilde{D}~~ \forall t \in \RR^+ \backslash \NN^*,
$$
this is a direct result of Lemma \ref{xtdiff} and that the function $x_{t'}$ is defined per piece on the intervals $[i,i+1]$.  So we get $\tilde{D}=C$ where $C$ is the constant of Lemma \ref{xtdiff}. To conclude, it is necessary to modify the family $x_t'$ to obtain a $\C^1$-interpolation at integers points. In order to do this, we remark, thanks to Lemma \ref{xtdiff}, that 
$$
\exists 1/4 > \delta>0,~~ \forall i \in \NN^*,~~ \forall t \in [i- \delta, i + \delta] ,~~  x_t' \in B(x_i',1/2).
$$
We define the polynomial function $P_i=(P^1_i, \cdots, P^r_i)$ where the function $P^j$ is defined by
\begin{align*}
P^j_i(t)&:= \left( -2 \left[ x_{i+ \delta}' - x_{i - \delta}' \right] + 2 \delta \, \left[\cfrac{dx_t'}{dt} \vert_{t =i- \delta} + \cfrac{dx_t'}{dt} \vert_{t =i + \delta} \right]  \right)_j \, t^3 \\
&+ \left( 3 \left[ x_{i+ \delta}' - x_{i - \delta}' \right] - 2 \delta \, \left[ 2 \, \cfrac{dx_t'}{dt} \vert_{t =i- \delta} + \cfrac{dx_t'}{dt} \vert_{t =i + \delta} \right] \right)_j \, t^2  \\
&+ \left( \cfrac{dx_t'}{dt} \vert_{t =i - \delta} \right)_j \, t +  \left( x_{i - \delta} \right)_j,
\end{align*}
where $\left( \cdot \right)_j$ refers to the $j$-th coordinate of the vector in the paranthesis. We note
\begin{equation*}
  \left\{
      \begin{aligned}
     P_i(0) &=x_{i- \delta}'\\
      P_i(1) &=x_{i+ \delta}' \\
      P_i'(0) &= 2 \delta \, \cfrac{dx_t'}{dt} \vert_{t =i- \delta} \\
      P_i'(1) &= 2 \delta \, \cfrac{dx_t'}{dt} \vert_{t =i+ \delta}. \\
      \end{aligned}
    \right.
    \end{equation*}
So we define the function $f_i(t) : [i- \delta, i + \delta]\rightarrow \aaa_1$ by
$$
f_i(t):= P_i \left( \dfrac{t - (i - \delta)}{2 \delta}\right),~~ \forall i \in \NN^*.
$$
We see that
\begin{equation*}
  \left\{
      \begin{aligned}
     f_i(i- \delta) &=x_{i- \delta}'\\
      f_i(i+ \delta) &=x_{i+ \delta}' \\
      f_i'(i- \delta) &= \cfrac{dx_t'}{dt} \vert_{t =i- \delta} \\
      f_i'(i+ \delta) &= \cfrac{dx_t'}{dt} \vert_{t =i+ \delta}. \\
      \end{aligned}
    \right.
    \end{equation*}
In addition, by choosing $\delta$, we have $\vert x_{i- \delta} - x_{i + \delta} \vert \leq 1$ and recall that the function $dx_t'/dt$ is uniformly bounded to non-integer points, we therefore obtain by that there are two constants $A$ and $\tilde{A}$ independent of $i$ and $t$ such that
\begin{equation*}
  \left\{
      \begin{aligned}
       \vert x_{t} - f_i(t) \vert &\leq A  & ~~ \forall t \in [i - \delta , i + \delta] \\
       \vert f_i'(t) \vert & \leq  \tilde{A}   & ~~ \forall t \in [i - \delta , i + \delta].
      \end{aligned}
    \right.
    \end{equation*}
To conclude, it is therefore sufficient to replace for any $t \in[i- \delta, i+ \delta]$ the $x_t'$ function by the $f_i(t)$ function. We then have the desired result.

We set
\begin{equation*}
  \left\{
      \begin{aligned}
      \tilde{u}_t:= u_t(x_t' + \cdot) \\
     \tilde{\varphi}_t := \tilde{u}_t - u_0. \\
      \end{aligned}
    \right.
\end{equation*}
Let us start by noting that the function $\tilde{\varphi}$ extends by density to a global function on $M$ which is the potential of $\exp(x_t)^* \omega_\varphi$ with respect to $\omega_0$ so $\omega_{\tilde{\varphi}}$ is a Kähler form. In addition, we have
$$
\vert \tilde{u} - \overline{\overline{u}} \vert \leq d_0 \vert x_t - x_{t}' \vert + m_t + c_t \leq C,
$$
and so
\begin{align*}
\Vert \tilde{\varphi}_t \Vert_{\C^0} &= \Vert \tilde{u}_t -u_0 \Vert_{\C^0} \\
&= \Vert \tilde{u}_t - \overline{\overline{u}}_t + \overline{\overline{u}}_t -u_0 \Vert_{\C^0} \\
&= \Vert \tilde{u}_t - \overline{\overline{u}}_t \Vert_{\C^0} + \Vert \overline{\overline{u}}_t -u_0 \Vert_{\C^0} \\
& \leq C
\end{align*}
Now, we notice that $\tilde{u}_t$ satisfies the following equation on $\aaa_1$ :
$$
\dfrac{\partial \tilde{u}_t}{ \partial t} = \log \det((\tilde{u}_t)_{ij}) + \tilde{X}(\tilde{u}_t) + \tilde{u}_t + \sum_{\alpha \in \Phi^+_P} \log \langle \alpha, \nabla \tilde{u}_t + 4 \tau_{\rho_P} \rangle,
$$
and by density, we then show that $\tilde{\varphi}$ satisfies the following equation on $M$ :
\begin{equation}\label{equationvarphitilde}
\dfrac{\partial \tilde{\varphi}_t}{\partial t} = \log \det((g_0)_{i \overline{j}} +(\tilde{\varphi}_t)_{i \overline{j}}) - \log \det((g_0)_{i \overline{j}}) + \tilde{X}_t(\tilde{\varphi}_t) - h + \theta_{\tilde{X}_t} +\tilde{\varphi}_t,
\end{equation}
où $\tilde{X}_t:=dx_t'/dt$ is the holomorphic vector field given by
$
\tilde{X}_t: x \in M \mapsto \partial/\partial s \vert_{s=0} \, \exp( s \, d x'_t/dt ) \cdot x .
$
Equation \eqref{equationvarphitilde} is the \textit{Kähler-Ricci flow modified by  $\tilde{X}_t$}. In particular, it can be written in the form
\begin{equation}\label{KRFM}
\cfrac{\partial \omega_{\tilde{\varphi}} }{\partial t} = - Ric(\omega_{\tilde{\varphi}_t}) + \omega_{\tilde{\varphi}_t} + \LLL_{\tilde{X}_t} \omega_{\tilde{\varphi}_t}.
\end{equation}

Let us finish this section, by noting that $\theta_{\tilde{X}}=\tilde{X}(u_0)$, we get, by Proposition \ref{c_t<C}, that there exists a constant $C>0$ independent of $t$ such that
$$
\vert \cfrac{\partial \tilde{\varphi}_t}{\partial t} \vert = \vert \cfrac{\partial \varphi_t}{\partial t} + \theta_{\tilde{X}_t} \vert \leq C.
$$

\subsubsection{Proof of the convergence}

We can now prove of the main theorem. Before we will prove the following result:
\begin{theorem}\label{cvss}
There exists a sequence of positives reals $(t_i)_{i \in \NN}$ such that $t_i \rightarrow_{i \rightarrow + \infty} + \infty$ and the subsequence $\tilde{\varphi}_{t_i}$, extracted to the solution $\tilde{\varphi}_t$ of the Käher Ricci flow modified by $X_t$ (Equation \eqref{KRFM}), converge to a potential $\overline{\varphi}_\infty$ such that $(\omega_0 + \sqrt{-1}/2\pi \, \partial \overline{\partial} \overline{\varphi}_{\infty}, X)$ is a Kähler-Ricci soliton where $X$ is the solitonic vector field defined by Proposition \ref{formX}.
\end{theorem}
\begin{proof}
The proof is divided into several points:
\begin{itemize}
\item[$\bullet$] Using the $\mathcal{C}^2$ estimates of Yau found in \cite{Y} for some Monge-Amp\`{e}re equations, we show that there is a constant $C_2>0$ independent of time $t$
$$
\Vert \tilde{\varphi}_t \Vert_{\mathcal{C}^2}  \leq C_2 \text{  et  } \left( (g_0)_{i \overline{j}} + (\tilde{\varphi}_t)_{i \overline{j}} \right)>0.
$$
Moreover, using Calabi's computations in \cite{Y}, we finally obtain that there is a time-independent constant $C_3>0$ such that:
$$
\Vert \tilde{\varphi}_t \Vert_{\C^3} \leq C_3.
$$
Finally, using the regularization theorem of parabolic equations, we obtain:
\begin{equation}\label{estimationtildevarphi}
\forall k \geq 0, ~~ \exists C_k \geq 0 ,~~ \Vert \tilde{\varphi}_t \Vert_{\C^k} \leq C_k.
\end{equation}
So for any sequence $t_n$ of real such that $t_i \rightarrow_{i \rightarrow + \infty} + \infty$, we can find an extracter $\gamma$ such as the subsequence $\tilde{\varphi}_{t_{\gamma(i)}}$ converges to a smooth function $\tilde{\varphi}_{\infty}$.
\item[$\bullet$] Recall that $X$ is the solitonic vector fiel defined in Proposition \ref{formX}, we denote $\sigma_t:= \exp(tX)$ is one parameter subgroup of automorphisms induced by $X$. We consider multiplying by $\exp(x_t)$ on $G/H$, because we have a action of $G$ on $M$, this application extends into an automorphism $\rho_t$ on $M$.
In addition, we can see directly that $\rho_t$ checks
$
\rho_t^* \tilde{\varphi}_t  = \varphi_t.
$
We set $\rho'_t = \rho_t \circ \sigma_t^{-1}$. We get
\begin{align*}
\int_M \Vert \overline{\partial} ( (\sigma_t')^* \cfrac{\partial \varphi_t'}{ \partial t} ) \Vert^2_{\omega_{\tilde{\varphi}_t}} e^{\theta_X + X(\overline{\varphi}_t)} \, \omega^n_{\overline{\varphi}_t} &= \int_M \Vert \overline{\partial} ( \cfrac{\partial \varphi_t'}{ \partial t} ) \Vert^2_{\omega_{\varphi_t'}} e^{\theta_X + X(\varphi_t')} \, \omega^n_{\varphi_t'}  \\
&= -\dfrac{d \tilde{\mu}_{\omega_0}(\varphi_t')}{dt}.
\end{align*}
In addition, we know that there is a uniform bound for $ \vert X(\psi) \vert$ for all $ \psi \in \mathcal{M}_X(\omega_0)$ (see Corollary 5.3 of \cite{Zh}) and that the function $\theta_X$ is bounded since continuous on a compact variety so we obtain that there exists a constant $C_1>0$ independent of $t$ such that
\begin{equation}\label{ineqfin}
0 \leq C_1 \, \int_M \Vert \overline{\partial} ( (\sigma_t')^* \cfrac{\partial \varphi_t'}{ \partial t} ) \Vert^2_{\omega_{\tilde{\varphi}_t}}  \, \omega^n_{\overline{\varphi}_t} \leq \int_M \Vert \overline{\partial} ( (\sigma_t')^* \cfrac{\partial \varphi_t'}{ \partial t} ) \Vert^2_{\omega_{\tilde{\varphi}_t}} e^{\theta_X + X(\overline{\varphi}_t)} \, \omega^n_{\overline{\varphi}_t} =  -\dfrac{d \tilde{\mu}_{\omega_0}(\varphi_t')}{dt}.
\end{equation}
In particular, since we have
$$
\tilde{C} \leq \tilde{\mu}_{\omega_0}(\varphi_t') \leq 0,
$$
we get that the function $ t \mapsto \tilde{\mu}_{\omega_0}(\varphi_t') $ is decreasing and so admits a finite limit at $+ \infty$. Moreover, there exits a sequence of reals $(t_i)_{i \in \NN}$ such that
\begin{equation}\label{lim=0}
\lim_{i \rightarrow + \infty} \, \int_M \Vert \overline{\partial} ( (\sigma_t')^* \cfrac{\partial \varphi_t'}{ \partial t}  \vert_{t=t_i}) \Vert^2_{\omega_{\tilde{\varphi}_{t_i}}}  \, \omega^n_{\overline{\varphi}_{t_i}}=0.
\end{equation}
Indeed, we define the sequence $(t_i)_{i \in \NN}$ such that
$$
\cfrac{d \, \tilde{\mu}_{\omega_0}(\varphi_t')}{dt} \vert_{t=t_i} = \sup_{s \in [i,i+1]} \left(  \cfrac{d \, \tilde{\mu}_{\omega_0}(\varphi_t')}{dt} \vert_{t=s} \right),
$$
and so, thanks to Equations \eqref{ineqfin}, its satsifies
$$
0 \leq C_1 \, \int_M \Vert \overline{\partial} ( (\sigma_t')^* \cfrac{\partial \varphi_t'}{ \partial t} \vert_{t=t_i}) \Vert^2_{\omega_{\tilde{\varphi}_{t_i}}}  \, \omega^n_{\overline{\varphi}_{t_i}} \leq  -\dfrac{d \tilde{\mu}_{\omega_0} (\varphi_t')}{dt} \vert_{t=t_i} \leq \tilde{\mu}_{\omega_0} (\varphi'_i)  - \tilde{\mu}_{\omega_0} (\varphi'_{i+1}) .
$$
We conlude by remarking that the right terme converge to $0$ when $i$ tends to $+\infty$ since $\tilde{\mu}_{\omega_0}(\varphi'_t)$ admits a finite limit when $t$ tends to $+ \infty$.
\item[$\bullet$] By using Equation \eqref{equationvarphitilde}, we get that
\begin{equation}\label{aaaa}
\cfrac{\sqrt{-1}}{2 \pi} \partial \overline{\partial} [ (\sigma_t')^* \cfrac{\partial \varphi_t'}{\partial t}] = - Ric(\omega_{\tilde{\varphi}_t}) + \omega_{\tilde{\varphi}_t} + \LLL_X \omega_{\tilde{\varphi}_t},
\end{equation}
it implies, thanks to the equation \ref{estimationtildevarphi}, that $(\sigma_t')^* \cfrac{\partial \varphi_t'}{\partial t} $ is uniformly bounded in a $C^k$ way for all $k \in \NN$. So according to Arzela-Ascoli's theorem, we can extract a convergent subsequence $(\sigma_t')^* \cfrac{\partial \varphi' }{\partial t}(t_i, \cdot) $. Moreover, by using Equation \eqref{lim=0} and Lemma \ref{lim0}, we obtain, even if it means extracting a subsequence again, that $(\sigma_t')^* \cfrac{\partial \varphi_t'}{\partial t}(t_i, \cdot) $ converges to $0$ in a $\C^k$ way for all $k \in \NN$. Using the equation \eqref{aaaa} and using the first point of the proof to extract a convergent subsequence, we obtain that $\omega_{\tilde{\varphi}_{t_i}}$ converges to the Kähler-Ricci soliton defined by
$$
(\omega_0 + \dfrac{\sqrt{-1}}{2 \pi} \partial \overline{\partial} \tilde{\varphi}_{\infty},X).
$$
\end{itemize}
\end{proof}
\begin{corollary}
There exists a Kähler-Ricci soliton $(g_{KRS},X_{KRS})$ on every horospherical manifold $M$.
\end{corollary}

If we use the result of uniqueness (theorem $3.2.$ of \cite{TZ2}), we obtain that there is an automorphism $\beta \in \operatorname{Aut}_r(M)$ such that
$$
\omega_{KRS}=\beta^* \omega_{\tilde{\varphi}_\infty}, ~~ X_{KRS}= \left(\beta^{-1} \right)_*(X).
$$
Then, for simplicity, we can assume that $\beta = Id$. Now, we can complete Theorem \ref{cvss}. 

\begin{theorem}
With the previous notation, the solution $\omega_{\tilde{\varphi}_t}$ to the Kähler-Ricci modified by $X_t$ converge to $\omega_{\tilde{\varphi}_\infty}$ when $t \rightarrow + \infty$.
\end{theorem}
\begin{proof}
To prove this theorem, thanks to Arzela-Ascoli theorem, it is sufficient to show that for any sequence of real positives $(t_i)_{n \in \NN}$ such that $t_i \rightarrow + \infty$, there is a subsequence of extractors $\gamma$ such as $\omega_{\tilde{\varphi}_{t_{\gamma(i) }}}$ converge to $\omega_{\tilde{\varphi}{\infty}}$. We consider a sequence of real $(t_i)_{i \in \NN}$ such that $t_i \rightarrow + \infty$ and set $\epsilon>0$. We are going to divide the proof into several points:
\begin{itemize}
\item[$\bullet$]  The first step is to show that there is a sequence of real $(s_i)_{i \in \NN}$ such that
$$
s_i \in[t_i - \epsilon, t_i + \epsilon],
$$
and such that
\begin{equation}\label{lim=0s_i}
\lim_{i \rightarrow + \infty} \, \int_M \Vert \overline{\partial} ( (\sigma_t')^* \cfrac{\partial \varphi_t'}{ \partial t}  \vert_{t=t_i}) \Vert^2_{\omega_{\tilde{\varphi}_{t_i}}}  \, \omega^n_{\overline{\varphi}_{t_i}}=0.
\end{equation}
In particular, by taking up the approach of the proof of Theorem \ref{cvss}, we also have, even if it means extracting, that
$
\omega_{\varphi_{s_i}} 
$
converges to $\omega_{\varphi_{s_\infty}}$ which is a Kähler-Ricci soliton and in particular
\begin{equation}\label{Juju}
\Vert \cfrac{\partial \varphi_t}{\partial t} \vert_{t=s_i} \Vert_{\C^k} \longrightarrow_{i \rightarrow + \infty} 0.
\end{equation}

Moreover, thanks to the theorem of uniqueness (Theorem $3.2.$ of \cite{TZ2}), there exists $\beta \in Aut_r(M)$ such that
$
\beta^*\omega_{\varphi_{s_{\infty}}} = \omega_{KRS}.
$
Moreover, thanks to uniqueness of Kähler-Ricci solitons ( Theorem $3.2.$ of \cite{TZ2}), we know that there exists $\beta \in Aut_r(M)$ such that
$
\beta^*\omega_{\varphi_{s_{\infty}}} = \omega_{KRS}.
$
For simplicity, even if it means changing $\omega_{KRS}$, we can assume that $\beta=Id$.
$$
\cfrac{d \, \tilde{\mu}_{\omega_0}(\varphi_t')}{dt} \vert_{t=s_i} = \sup_{s \in [t_i- \epsilon,t_i+ \epsilon]} \left(  \cfrac{d \, \tilde{\mu}_{\omega_0}(\varphi_t')}{dt} \vert_{t=s} \right),
$$
and, thanks to Equation \eqref{ineqfin},
$$
0 \leq C_1 \, \int_M \Vert \overline{\partial} ( (\sigma_t')^* \cfrac{\partial \varphi_t'}{ \partial t} \vert_{t=s_i}) \Vert^2_{\omega_{\tilde{\varphi}_{s_i}}}  \, \omega^n_{\overline{\varphi}_{s_i}} \leq  -\dfrac{d \tilde{\mu}_{\omega_0} (\varphi_t')}{dt} \vert_{t=s_i} \leq \cfrac{1}{\epsilon} \left( \tilde{\mu}_{\omega_0} (\varphi'_{t_i - \epsilon })  - \tilde{\mu}_{\omega_0} (\varphi'_{t_i + \epsilon}) \right).
$$
We then conclude by noting that the right term converges to $0$ when $i$ tends to
$+\infty$ since $\tilde{\mu}_{\omega_0}(\varphi'_t)$ admits a finite limit when $t$ tends to $+ \infty$.

\item[$\bullet$] The second step is to show that the sequence
$(\tilde{\varphi}_{t_i} - \tilde{\varphi}_{s_i} )_{i \in \NN}$ checks from a certain rank that
$$\Vert\tilde{\varphi}_{t_i} - \tilde{\varphi}_{s_i} \Vert_{\C^k} \leq \delta( \epsilon),
$$
where $\delta(\epsilon)$ satisfies $\delta(\epsilon)\longrightarrow_{\epsilon \rightarrow 0} 0.
$
To do this, we will use the theory of parabolic equations. Indeed, by using Equation \eqref{equationvarphitilde}, we prove that  $\left(  \tilde{\varphi}_t - \tilde{\varphi}_{s_i}  \right)$ is solution for $t \in [t_i - \epsilon , t_i + \epsilon]$ of Monge-Ampère flow : 
$$
\dfrac{\partial \left(  \tilde{\varphi}_t - \tilde{\varphi}_{s_i}  \right)}{\partial t} = L(t,x) \left ( \tilde{\varphi}_t - \tilde{\varphi}_{s_i}  \right)  - G(t,x),
$$
où
$$
L(t,x) \varphi := \log \left( \cfrac{ \left( \omega_{\tilde{\varphi}_{s_i}} + \sqrt{-1} \partial \overline{\partial} \, \varphi \right)^n} { \left( \omega_{\tilde{\varphi}_{s_i}} \right)^n} \right) + \tilde{X}_t \left( \varphi \right) + \varphi
$$
and
\begin{align*}
G(t,x) :&= \left( \tilde{X}_t - \tilde{X}_{s_i} \right) \left( \tilde{\varphi}_{s_i} \right) - \left(  \theta_{\tilde{X}_t} - \theta_{\tilde{X}_{s_i}} \right) - \cfrac{\partial \varphi_t}{\partial t} \vert_{t=s_i}.
\end{align*}
We notice that from a certain rank, thanks to Equation \eqref{condxt'} and \eqref{Juju}, we have that there exists a constant $C>0$ independent of $t$ such as $\Vert G \Vert_{\C^k} \leq C \epsilon$. To conclude, it is then sufficient to use the theorem of implicit functions.
\item[$\bullet$] We then conclude by showing that $\tilde{\varphi}_{t_i}$ converges to $\varphi_{KRS}$ in a $\C^k$ way for all $k>0$. Indeed, we have
$$
\Vert \tilde{\varphi}_{t_i} - \varphi_{KRS} \Vert_{\C^k} \leq \Vert \tilde{\varphi}_{t_i} - \tilde{\varphi}_{s_i}\Vert_{\C^k} + \Vert  \tilde{\varphi}_{s_i} -  \varphi_{KRS} \Vert_{\C^k}.
$$
Using the two previous points, we obtain that from a certain rank, we have
$$
\Vert \tilde{\varphi}_{t_i} - \varphi_{KRS} \Vert_{\C^k} \leq \delta(\epsilon) + \epsilon.
$$
Since the result is true for any $\epsilon>$0, we get the desired result.
\end{itemize}
\end{proof}

By definition of convergence in the Cheeger-Gromov sense, we directly obtain the main theorem that we recall here

\begin{theorem}\label{cvflotHS}
Let $M$ a Fano horosphercal manifold of horospherical homegeneous space $G/H$ where $G$ is the complexification of a maximal compact subgroup $K$. The solution $\omega_{\varphi_t}$ defined for $t \in [0, + \infty[$ of the Kähler-Ricci flow with inital value a $K$-invariant metric $\omega_0$ converge in the Cheeger-Gromov sense to a Kähler form $\omega_{\infty}$ when $t$ tends to $+ \infty$ where $\omega_{\infty}$ is a Kähler-Ricci soliton.
\end{theorem}
\bibliographystyle{alpha}
\bibliography{biblio}
\end{document}